\documentclass{article}

\usepackage{amssymb,amsmath,amsthm}
\usepackage{latexsym}

\def\er{{\mathbb R}}
\def\Ex{{\mathbb E}}

\def\N{{\mathbb N}}
\def\IC{\mathrm{IC}}
\def\CI{\mathrm{CI}}
\def\CWSM{\mathrm{CWSM}}
\def\Med{\mathrm{Med}}

\def\R{{\er}}
\def\mi{{\mu}}
\def\ra{{\rightarrow}}
\def\is#1#2{\left\langle #1 , #2 \right\rangle}
\def\BB{\mathcal{B}}
\def\cpn{{r_{p,n}}}

\def\1{{\mathbf{1}}}

\def\LL{\mathcal{L}}

\def\npn{{\nu_p^n}}
\def\n2n{{\nu_2^n}}
\def\nn{{\nu^n}}
\def\mpn{{\mu_{p,n}}}

\def\np1{{\nu_p}}

\def\cppb2{{C_5}}

\newtheorem{thm}{Theorem}[section]
\newtheorem{prop}[thm]{Proposition}
\newtheorem{lem}[thm]{Lemma}
\newtheorem{cor}[thm]{Corollary}
\newtheorem{remark}[thm]{Remark}

\newtheorem{conj}{Conjecture}
\newtheorem{Def}[thm]{Definition} %%\numberwithin{equation}{subsection}

\title{On the infimum convolution inequality
\thanks{Partially supported by the Foundation for Polish Science and MEiN Grant 1 PO3A 012 29}
}
\author{R.\ Lata{\l}a and J.\ O.\ Wojtaszczyk}
\date{}

\begin{document}

\maketitle

\begin{abstract}
In the paper we study the {\em infimum convolution} inequalites. Such an inequality was
first introduced by B. Maurey to give the optimal concentration of measure behaviour for
the product exponential measure. We show how {\em IC} inequalities are tied to concentration and
study the optimal cost functions for an arbitrary probability measure $\mi$. In particular,
we show the optimal {\em IC} inequality for product log--concave measures and for uniform
measures on the $\ell_p^n$ balls. Such an optimal inequality implies, for a given measure,
in particular the Central Limit Theorem of Klartag and the tail estimates of Paouris.
\end{abstract}

\section{Introduction and Notation}

In the seminal paper \cite{Ma}, B.\ Maurey introduced the so called
property $(\tau)$ for a probability measure $\mi$ with a cost function $\varphi$ 
(see Definition \ref{propertyTauDefinition} below)
and established a very elegant and simple proof of Talagrand's
two level concentration for the product exponential distribution $\nu^n$ using $(\tau)$ for this
distribution and an appropriate cost function $w$.

It is natural to ask what other pairs $(\mu,\varphi)$ have 
property $(\tau)$? As any $\mu$ satisfies $(\tau)$ with $\varphi \equiv 0$,
one will rather ask how big a cost function can one take.
 In this paper we study the probability measures $\mu$
that have property $(\tau)$ with respect  to the largest
(up to a multiplicative factor) possible convex cost function $\Lambda^\star_\mu$.
This bound comes from checking property $(\tau)$ for linear
functions. We say a measure satisfies the {\em infimum convolution inequality}
($IC$ for short) if the pair $(\mu, \Lambda^\star_\mu)$ satisfies $\tau$.

It turns out that such an optimal infimum convolution inequality has very strong consequences.
It gives the best possible concentration behaviour, governed by the so--called $L_p$-centroid bodies (Corollary \ref{IC_and_CI}).
This, in turn, implies in particular a
weak--strong moment comparison (Proposition \ref{weakstr}), the Central Limit Theorem of Klartag \cite{Kl} and the tail estimates
estimates of Paouris \cite{Pa} (Proposition \ref{conseq}). We believe that $IC$ holds for
any log--concave probability measure, which is the main motivation for this paper.

Maurey's inequality for the exponential measure is of this optimal type.
We transport this to any log--concave measure on the real line, and as the inequality
tensorizes, any product log--concave measure satisfies $IC$ (Corollary \ref{SIClogprod}). 
However, the main challenge is to provide non--product examples of measures satisfying $IC$. 
We show how such an optimal result can be obtained from concentration inequalites, and follow
on to prove $IC$ for the uniform measure on any $\ell_p^n$ ball for $p \geq 1$ (Theorem \ref{glownyWynik}).

With the techniques developed we also prove a few other results. We give a proof
of the log--Sobolev inequality for $\ell_p^n$ balls, where $p \geq 2$ (Theorem \ref{logSobolev}) 
and provide a new concentration inequality for the exponential measure for sets lying far away from 
the origin (Theorem \ref{PushAndPop}).

{\bf Organization of the paper.}
This section, apart from the above introduction, defines the notation used
throughout the paper. The second section is devoted to studying the general
properties of the inequality $IC$. In subsection 2.1 we recall the definition
of property $(\tau)$ and its ties to concentration from \cite{Ma}. In subsection
2.2 we study the opposite implication --- what additional assumptions one needs
to infer $(\tau)$ from concentration inequalities. In subsection 2.3 we show
that $\Lambda^\star_\mi$ is indeed the largest possible cost function and 
define the inequality $IC$. In subsection 2.4 we show that product log--concave
measures satisfy $IC$.

In the third section we give more attention to the concentration inequalities
tied to $IC$. In subsection 3.1 we show the connection to $\mathcal{Z}_p$ bodies.
In subsection 3.2 we continue in this vein with the additional assumption our
measure is $\alpha$--regular. In subsection 3.3 we show how $IC$ implies a comparison
of weak and strong moments and the results of \cite{Kl} and \cite{Pa}.

In the fourth section we give a modification of the two--level concentration for the exponential measure, in which
for sets lying far away from the origin only an enlargement by $t B_1^n$ is used. This
will be used in the fifth section, which focuses on the uniform measure on the $B_p^n$ ball. In 
subsection 5.1 we define and study two rather standard transports of measure used further on.
In subsection 5.2 we use these transports along with the concentration from section 4 and
a Cheeger inequality from \cite{So} to give a proof of $IC$ for $p \leq 2$. In section 5.3
we show a proof of $IC$ for $p \geq 2$ and a proof of the log--Sobolev inequality for $p \geq 2$.

We conclude with a few possible extensions of the results of the paper in the sixth section.
\medskip

{\bf Notation.} By $\langle\cdot,\cdot\rangle$ we denote the standard scalar
product on $\er^n$. For $x\in \er^n$ we put
$\|x\|_p=(\sum_{i=1}^n|x_i|^p)^{1/p}$ for $1\leq p< \infty$ and
$\|x\|_{\infty}=\max_{i}|x_i|$, we also use $|x|$ for $\|x\|_2$.
We set $B_p^n$ for a unit ball in $l^n_p$, i.e..\
$B_p^n=\{x\in\er^n\colon \|x\|_p\leq 1\}$.

By $\nu$ we denote the symmetric exponential distribution on $\er$,
i.e.\ the probability measure with the density $\frac{1}{2}\exp(-|x|)$.
For $p\geq 1$, $\nu_p$ is the probability distribution on $\er$ with 
the density $(2\gamma_p)^{-1}\exp(-|x|^p)$, where $\gamma_p=\Gamma(1+1/p)$, 
in particular $\nu_1=\nu$. For a probability measure $\mu$ we write $\mu^n$
for a product measure $\mu^{\otimes n}$, thus $\nu_p^n$ has the 
density $(2\gamma_p)^{-n}\exp(-\|x\|_p^p)$.

For a Borel set $A$ in $\er^n$ by $|A|$ or $\lambda_n(A)$ we mean the Lebesgue measure of $A$.
We choose numbers $\cpn$ in such a way that $|\cpn B_p^n|=1$ and
by $\mu_{p,n}$ denote the uniform distribution on $B_p^n$.

The letters $c, C$ denote absolute numerical constants, which may change from line to line. By $c(p), C(p)$ we mean constants dependent on $p$ (or, formally, a family of absolute constants indexed by $p$), these also may change from line to line. Other letters, in particular greek letters, denote constants fixed for a given proof or section. For any sets of positive real numbers $a_i$ and $b_i$, $i \in I$, by $a_i \sim b_i$ we mean there exist absolute numerical constants $c, C > 0$ such that $ca_i < b_i < Ca_i$ for any $i \in I$. Similarly for collections of sets $A_i$ and $B_i$ by $A_i \sim B_i$ we mean $c A_i \subset B_i \subset C A_i$ for any $i \in I$, where again $c, C > 0$ are absolute numerical constants. By $\sim_p$ we mean the constants above can depend on $p$.

\section{Infimum convolution inequality}

\subsection{Property $(\tau)$}

\begin{Def} 
\label{propertyTauDefinition}
Let $\mu$ be a probability measure on $\er^n$ and
$\varphi\colon\er^n\rightarrow [0,\infty]$ be a measurable
function. We say that the pair
$(\mu,\varphi)$ has {\em property $(\tau)$} if
for any bounded  measurable function $f\colon\er^n\rightarrow \er$,
\begin{equation}
\label{in_infconv}
\int_{\er^n}e^{f\Box \varphi}d\mu\int_{\er^{n}}e^{-f}d\mu\leq 1, 
\end{equation}
where for two functions $f$ and $g$ on $\er^n$,
$$ 
f\Box g(x):=\inf\{f(x-y)+g(y)\colon y\in\er^n\} 
$$
denotes the {\em infimum convolution} of $f$ and $g$.
\end{Def}

The following two easy observations are almost immediate (c.f.\ \cite{Ma}):

\begin{prop}[Tensorization]
\label{tensMaur}
If pairs $(\mu_i,\varphi_i)$, $i=1,\ldots,k$ have property $(\tau)$  and
$\varphi(x_1,\ldots,x_k)=\varphi_1(x_1)+\ldots+\varphi_k(x_k)$,
then the couple $(\otimes_{i=1}^{k}\mu_i,\varphi)$ also has property $(\tau)$.
\end{prop}

\begin{prop}[Transport of measure]
\label{tau_trans}
Suppose that $(\mu,\varphi)$ has  property $(\tau)$ and
$T\colon \er^n\rightarrow\er^m$ is such that
$$ \psi(Tx-Ty)\leq \varphi(x-y) \mbox{ for all } x,y\in \er^n. $$ 
Then the pair $(\mu\circ T^{-1},\psi)$ has property $(\tau)$.
\end{prop}

Maurey noticed that property $(\tau)$ implies
$\mu(A+B_{\varphi}(t))\geq 1-\mu(A)^{-1}e^{-t},$
where
$$
B_{\varphi}(t):=\{x\in \er^n\colon \varphi(x)\leq t\}.
$$
We will need a slight modification of this estimate.

\begin{prop}
\label{ICconc}
Property $(\tau)$ for $(\varphi,\mu)$ implies for any Borel set
$A$ and $t\geq 0$,
\begin{equation}
\label{conc} 
\mu(A+B_{\varphi}(t))\geq \frac{e^t\mu(A)}{(e^t-1)\mu(A)+1}.
\end{equation}
In particular for all $t>0$,
\begin{align}
\label{concsmall}
\mu(A)>0&\ \Rightarrow\
\mu(A+B_{\varphi}(t))> \min\{e^{t/2}\mu(A),1/2\},
\\
\label{conclarge}
\mu(A)\geq 1/2 &\ \Rightarrow\
1-\mu(A+B_{\varphi}(t))<e^{-t/2}(1-\mu(A))
\end{align}
and 
\begin{equation}
\label{concall}
\mu(A)=\nu(-\infty,x]\ \Rightarrow\
\mu(A+B_{\varphi}(t))\geq \nu(-\infty,x+t/2].
\end{equation}
\end{prop}

\begin{proof}
Take $f(x) = t \1_{\R^n \setminus A}$. Then $f(x)$ is non--negative on
$\R^n$, so $f\Box \varphi$ is non--negative (recall that by definition we
consider only nonnegative cost functions). For $x \not\in A + B_\varphi(t)$
we have $f \Box \varphi (x) = \inf_y(f(y) + \varphi(x-y))\geq t$, 
for either
$y \not \in A$, and then $f(y) = t$, or $y \in A$, and then $\varphi(x-y)
\geq t$ as $x \not \in A + B_\varphi(t)$.

Thus from property $(\tau)$ for $f$ we have
\begin{align*} 
1 & \geq \int e^{f\Box \varphi (x)} d\mi(x) \int e^{-f(x)}
d\mi(x) 
\\ 
& \geq
\Big[\mi\big(A + B_\varphi(t)\big) + 
e^t\big(1 - \mi(A + B_\varphi(t))\big)\Big]\big[\mi(A)
+ e^{-t} (1-\mi(A))\big],
\end{align*}
from which, extracting the condition upon $\mi(A +B_\varphi(t))$ by direct
calculation, we get (\ref{conc}).

Let $f_t(p):=e^tp/((e^t-1)p+1)$, notice that
$f_t$ is increasing in $p$ and for $p\leq e^{-t/2}/2$,
$$
(e^{t}-1)p+1\leq e^{t/2}+1-\frac{1}{2}(e^{t/2}+e^{-t/2})<e^{t/2},
$$
hence $f_t(p)>\min(e^{t/2}p,1/2)$ and (\ref{concsmall}) follows.
Moreover for $p\geq 1/2$
$$
1-f_t(p)=\frac{1-p}{(e^t-1)p+1}\leq
\frac{1-p}{(e^t+1)/2}<e^{-t/2}(1-p)
$$
and we get (\ref{conclarge}).

Let $F(x)=\nu(-\infty,x]$ and $g_t(p)=F(F^{-1}(p)+t)$.
Previous calculations show that for $t,p>0$, 
$f_t(p)\geq g_{t/2}(p)$ if $F^{-1}(p)+t/2\leq 0$ or 
$F^{-1}(p)\geq 0$. Since $g_{t+s}=g_t\circ g_s$ and
$f_{t+s}=f_t\circ f_s$, we get that $f_t(p)\geq g_{t/2}(p)$ for
all $t,p>0$, hence (\ref{conc}) implies (\ref{concall}).  
\end{proof}

The main theorem of \cite{Ma} states that $\nu$ satisfies  ($\tau$) with
a sufficiently chosen cost function.

\begin{thm}
\label{inf_nu}
Let $w(x)=\frac{1}{36}x^2$ for $|x|\leq 4$ and $w(x)=\frac{2}{9}(|x|-2)$
otherwise. Then the pair $(\nu^n,\sum_{i=1}^{n}w(x_i))$ has  property 
$(\tau)$.
\end{thm}

Theorem \ref{inf_nu} together with Proposition \ref{ICconc} immediately
gives the following two-level concentration:
\begin{equation}
\label{conc_nu}
\nu^{n}(A)=\nu(-\infty,x]\ \Rightarrow\ 
\forall_{t\geq 0}\ \nu^{n}(A+6\sqrt{2t}B_2^n+18tB_1^n)
\geq \nu(-\infty,x+t],
\end{equation}
that was first established (with different universal, rather large
constants) by Talagrand \cite{Ta}.

\subsection{From concentration to property $(\tau)$}

Proposition \ref{ICconc} shows that property $(\tau)$ implies
concentration, the next result presents the first approach to the converse
implication.

\begin{cor}
\label{firstappr}
Suppose that the cost function $\varphi$ is radius-wise nondecreasing,
$\mu$ is a Borel probability measure on $\er^n$ and $\beta>0$ is such that
for any $t>0$ and $A\in {\cal B}(\er^n)$,
\begin{equation}
\label{mu_and_nu}
\mu(A)=\nu(-\infty,x] \Rightarrow
\mu(A+\beta B_{\varphi}(t))\geq \nu(-\infty,x+\max\{t,\sqrt{t}\}].
\end{equation}
Then the pair $(\mu,\frac{1}{36}\varphi(\frac{\cdot}{\beta}))$ 
has property $(\tau)$.
In particular if $\varphi$ is convex, symmetric and $\varphi(0)=0$
then (\ref{mu_and_nu}) implies property $(\tau)$
for $(\mu,\varphi(\frac{\cdot}{36\beta}))$.
\end{cor}

\begin{proof}
Let us fix $f\colon \er^n \rightarrow \er$. For any
measurable function $h$ on $\er^k$ and $t\in\er$ we put
$$
A(h,t):=\{x\in \er^k\colon h(x)< t\}.
$$
Let $g$ be a nondecreasing right-continuous 
function on $\er$ such that $\mu(A(f,t))=\nu(A(g,t))$. Then the distribution 
of $g$ with respect to $\nu$ is the same as the distribution of $f$ with 
respect to $\mu$ and thus
$$
\int_{\er^n}e^{-f(x)}d\mu(x)=\int_{\er}e^{-g(x)}d\nu(x).
$$
To finish the proof of the first assertion, 
by Theorem \ref{inf_nu} it is enough to show
that
$$ 
\int_{\er^n}e^{f\Box \frac{1}{36}\varphi(\frac{\cdot}{\beta})}d\mu\leq
\int_\er e^{g \Box w}d\nu.
$$
We will establish stronger property:
$$
\forall_{u}\ \mu\bigg(A\Big(f\Box
\frac{1}{36}\varphi\Big(\frac{\cdot}{\beta}\Big),u\Big)\bigg)
\geq \nu(A(g\Box w,u)).
$$
Since the set $A(g\Box w,u)$ is a halfline, it is enough to prove that
\begin{equation}
\label{aim1}
g(x_1)+w(x_2)<u\ \Rightarrow \
\mu\bigg(A\Big(f\Box \frac{1}{36}\varphi\Big(\frac{\cdot}{\beta}\Big),u\Big)\bigg)
\geq \nu(-\infty,x_1+x_2].
\end{equation}
Let us fix $x_1$ and $x_2$ with $g(x_1)+w(x_2)<u$ and take $s_1>g(x_1)$
$s_2=w(x_2)$ with $s_1+s_2<u$. Put $A:=A(f,s_1)$, then
$\mu(A)=\nu(A(g,s_1))\geq \nu(-\infty,x_1]$. By the definition of
$w$ it easily follows that $x_2\leq \max\{6\sqrt{s_2},9 s_2\}$,
hence by (\ref{mu_and_nu}), $\mu(A+\beta B_{\varphi}(36s_2))
\geq \nu(-\infty,x_1+x_2].$
Since
$$
A+\beta B_{\varphi}(36s_2)=
A(f,s_1)+B_{\varphi(\frac{\cdot}{\beta})/36}(s_2)
\subset A\bigg(f\Box \frac{1}{36}\varphi
\Big(\frac{\cdot}{\beta}\Big),s_1+s_2\bigg),
$$
we obtain the  property (\ref{aim1}).

The last part of the statement immediately follows since any
symmetric convex function $\varphi$ is radius-wise
nondecreasing and if additionally $\varphi(0)=0$, then
$\varphi(x/36)\leq \varphi(x)/36$ for any $x$.
\end{proof}

The next proposition shows that inequalities (\ref{concsmall}) and
(\ref{conclarge}) are strongly related.

\begin{prop}
\label{equiv}
The following two conditions are equivalent for any Borel set $K$ and
$\gamma>1$,\\
\begin{align}
\label{cond1}
\forall_{A\in {\cal B}(\er^n)}\ 
 &\mu(A)>0 \Rightarrow
  \mu(A+K)> \min\Big\{\gamma\mu(A),\frac{1}{2}\Big\},
\\
\label{cond2}
\forall_{\tilde{A}\in {\cal B}(\er^n)}\
  & \mu(\tilde{A})\geq \frac{1}{2} \Rightarrow
1-\mu(\tilde{A}-K)<\frac{1}{\gamma}(1-\mu(\tilde{A})).
\end{align}
\end{prop}

\begin{proof} 
(\ref{cond1})$\Rightarrow$(\ref{cond2}). Suppose that
$\mu(\tilde{A})\geq 1/2$ and
$1-\mu(\tilde{A}-K)\geq \gamma^{-1}(1-\mu(\tilde{A}))$.
Let $A:=\er^n\setminus (\tilde{A}-K)$,
then $(A+K)\cap \tilde{A}=\emptyset$, so
$\mu(A+K)\leq 1/2$ and
$$
\mu(A+K)\leq 1-\mu(\tilde{A})\leq \gamma(1-\mu(\tilde{A}-K))=
\gamma\mu(A)
$$ 
and this contradicts (\ref{cond1}).

(\ref{cond2})$\Rightarrow$(\ref{cond1}). Let us take $A\in \er^n$ with
$\mu(A)>0$
such that $\mu(A+K)\leq \min\{\gamma\mu(A),1/2\}$. 
Let $\tilde{A}:=\er^n\setminus (A+K)$,
then $\mu(\tilde{A})\geq 1/2$. Moreover $(\tilde{A}-K)\cap A=\emptyset$,
thus
$$
1-\mu(\tilde{A}-K)\geq \mu(A)\geq \frac{1}{\gamma}\mu(A+K)=
\frac{1}{\gamma}(1-\mu(\tilde{A}))
$$
and we get the contradiction with (\ref{cond2}). 
\end{proof}

\begin{cor}
\label{large_t}
Suppose that $t>0$ and $K$ is a symmetric convex set in $\er^n$
such that
$$
\forall_{A\in {\cal B}(\er^n)}\ \mu(A)>0 \Rightarrow
\mu(A+K)>\min\{e^t\mu(A),1/2\}.
$$
Then for any Borel set $A$,
$$
\mu(A)=\nu(-\infty,x]\ \Rightarrow
\mu(A+2K)> \nu(-\infty,x+t].
$$
\end{cor}

\begin{proof}
Let us fix the set $A$ with $\mu(A)=\nu(-\infty,x]$. Notice that
$A+2K=A+K+K\supset A+K$.
If $x+t\leq 0$, then $\mu(A+K)>e^t\mu(A)=\nu(-\infty,x+t]$.
If $x\geq 0$,
Proposition \ref{equiv} gives
$$
\mu(A+K)>1-e^{-t}(1-\mu(A))=\nu(-\infty,x+t].
$$
Finally, if $x\leq 0\leq x+t$, we get $\mu(A+K)\geq 1/2=\nu(-\infty,0]$,
hence by the previous case,
$$
\mu(A+2K)=\mu((A+K)+K)> \nu(-\infty,t]\geq \nu(-\infty,x+t].
$$
\end{proof}

Corollary \ref{large_t} shows that if the cost function
$\varphi$ is symmetric and convex, condition (\ref{mu_and_nu}) (with
$2\beta$ instead of $\beta$) for $t\geq 1$ is implied by the following: 
\begin{align}
\label{concsm}
\forall_{A\in{\cal B}(\er^n)}\ \mu(A)>0\ \Rightarrow\
\mu(A+\beta B_{\varphi}(t))>\min\{e^t\mu(A),1/2\}.
\end{align}

To treat the case $t\leq 1$ we will need Cheeger's version of the Poincar\'e
inequality.

We say that a probability measure $\mu$ on $\er^n$ satisfies
{\em Cheeger's inequality} with constant $\kappa$ if for any
Borel set $A$ 
\begin{equation}\label{cheeger}
\mu^{+}(A):=\liminf_{t\rightarrow 0+}\frac{\mu(A+tB_2^n)-\mu(A)}{t}
\geq \kappa \min \{\mu(A),1-\mu(A)\}.
\end{equation}
It is not hard to check that Cheeger's inequality (cf.\ \cite[Theorem 2.1]{BH2})
implies
$$
\mu(A)=\nu(-\infty,x]\ \Rightarrow\
\mu(A+tB_2^n)\geq \nu(-\infty,x+\kappa t].
$$

Finally, we may summarize this section with the following statement. 

\begin{prop}
\label{reduction}
Suppose that the cost function $\varphi$ is convex, symmetric with
$\varphi(0)=0$ and
$1\wedge\varphi(x)\leq (\alpha|x|)^2$ for all $x$. If 
the measure $\mu$ satisfies
Cheeger's inequality with the constant $\beta = 1\slash \delta$ and 
the condition (\ref{concsm}) is
satisfied for all $t\geq 1$ and $C=\gamma$ then $(\mu, \varphi(\cdot/C))$ 
has property $(\tau)$ with the constant
$C=36 \min\{2\gamma, \alpha\delta\}$.
\end{prop}

\begin{proof}
Notice that $\alpha B_{\varphi}(t)\supset \sqrt{t}B_2^n$ for all $t<1$, hence
Cheeger's inequality implies that condition (\ref{mu_and_nu})
holds for $t<1$ with $C=\alpha\delta$. Therefore (\ref{mu_and_nu})
holds for all $t\geq 0$ with $C=\min\{2\gamma, \alpha\delta\}$ and the
assertion follows by Corollary \ref{firstappr}.
\end{proof}

\subsection{Optimal cost functions}\label{Optimal_Cost_Subsection}

A natural question arises: what other pairs $(\mu,\varphi)$ have property
($\tau$)? First we have to choose the right cost function. To do this
let us recall the following definitions.

\begin{Def} 
Let $f : \R^n \ra (-\infty,\infty]$. The {\em Legendre transform of $f$}, 
denoted $\LL f$ is defined by $\LL f(x) := \sup_{y \in \R^n} \{\is{x}{y} - f(y)\}$.
\end{Def}

The Legendre transform of any function is a convex function. 
If $f$ is convex and lower semi-continuous, then $\LL\LL f = f$, 
and otherwise $\LL \LL f \leq f$. In general, if $f \geq g$, then 
$\LL f \leq \LL g$. The Legendre transform satisfies 
$\LL (Cf) (x) = C \LL f (x\slash C)$ and if $g(x) = f(x\slash C)$, 
then $\LL g(x) = \LL f(Cx)$. For this and other properties of $\LL$,
cf.\ \cite{MT}. The Legendre transform has been previously used in the
context of convex geometry, see for instance \cite{AKM} and \cite{KM}.

\begin{Def} 
Let $\mi$ be a probability measure on $\er^n$. We define
$$ 
M_\mi(v) := \int_{\R^n} e^{\is{v}{x}} d\mi(x), 
\quad \Lambda_\mi(v) := \log M_\mi(v)
$$ 
and 
$$ 
\Lambda^{\star}_\mi(v) := \LL \Lambda_\mi(v)
=\sup_{u\in \er^n}\Big\{\is{v}{u}
-\ln\int_{\er^n} e^{\is{u}{x}}d\mu(x)\Big\}. 
$$ 
\end{Def}

The function $\Lambda^{\star}_\mi$ plays a crucial role in the theory of large
deviations cf. \cite{DS}.

\begin{remark} 
Let $\mi$ be a symmetric probability measure on $\R^n$ and let
$\varphi$ be a convex cost function such that $(\mi,\varphi)$ 
has property ($\tau$). 
Then 
$$
\varphi(v) \leq 2\Lambda^\star_\mi (v\slash 2)\leq 
\Lambda^\star_\mi(v).
$$
\end{remark}

\begin{proof}
Take $f(x) = \is{x}{v}$. Then
$$
f \Box \varphi(x) = \inf_y (f(x-y) + \varphi(y)) = 
\inf_y (\is{x-y}{v} + \varphi(y)) = \is{x}{v} - \LL \varphi(v).
$$ 
Property $(\tau)$ yields
$$
1 \geq \int e^{f\Box \varphi} d\mi \int e^{-f} d\mi = 
e^{-\LL \varphi(v)} \int e^{\is{x}{v}} d\mi \int e^{-\is{x}{v}}d\mi = 
e^{-\LL \varphi(v)} M^2_\mi(v),
$$
where the last equality uses the fact that $\mi$ is symmetric. 
Thus by taking the logarithm we get $\LL \varphi(v) \geq 2 \Lambda_\mi(v)$,
 and by applying the Legendre transform we obtain 
$\varphi(v) = \LL \LL \varphi(v) \leq 2 \Lambda^\star_\mi(v\slash 2).$
The inequality $2\Lambda^\star_\mi (v\slash 2)\leq \Lambda^\star_\mi(v)$
follows by the convexity of $\Lambda_\mu^\star$.
\end{proof}

The above remark motivates the following definition.

\begin{Def}
We say that a symmetric probability measure $\mu$ satisfies the
{\em infimum convolution inequality
with constant $\beta$} ($\IC(\beta)$ in short), if the pair 
$(\mu,\Lambda_{\mu}^{*}(\frac{\cdot}{\beta}))$ has property $(\tau)$.
\end{Def}

\begin{prop} 
\label{product_IC}
If $\mi_i$  are symmetric probability measures on $\R^{n_i}$,
$1\leq i\leq k$  satisfying $\IC(\beta_i)$, then $\mu=\otimes_{i=1}^k \mi_i$ 
satisfies $\IC(\beta)$ with $\beta=\max_i \beta_i$.
\end{prop}

\begin{proof} By independence,
$\Lambda_{\mu}(x_1,\ldots,x_k)=\sum_{i=1}^{k}\Lambda_{\mu_i}(x_i)$ and
$\Lambda_{\mu}^*(x_1,\ldots,x_k)=\sum_{i=1}^{k}\Lambda_{\mu_i}^*(x_i)$.
Since $\IC(\beta)$ implies $\IC(\beta')$ with any $\beta'\geq \beta$, 
the result immediately
follows by Proposition \ref{tensMaur}.
\end{proof}

\begin{prop} 
For $v = (v_0,v_1,\ldots,v_n)$ in $\R^{n+1}$ let $\tilde{v}$
denote the vector $(v_1,v_2,\ldots,v_n) \in \R^n$. 
A probability measure $\mu$ on $\R^n$ satisfies $\IC(\beta)$ if and only if
for any nonempty  $V \subset \R^{n+1}$ and a bounded measurable function
$f$ on $\er^n$,
\begin{equation}
\label{V_IC}
\int_{\R^n} e^{f \Box \psi_V} d\mi \int_{R^n} e^{-f} d\mi \leq 
\sup_{v\in V} \Big(e^{v_0} \int_{\R^n} e^{\beta \is{x}{\tilde{v}}} d\mi(x)\Big),
\end{equation}
where
\[
\psi_V(x) := \sup_{v \in V}\{v_0 + \is{x}{\tilde{v}}\}.
\]
\end{prop}

\begin{proof}
If we put $V = \{(v_0,\tilde{v})\colon v_0 = -\Lambda_\mi(\beta \tilde{v})\}$, 
then the right-hand side is equal to 1 and 
$\psi_V(x) = \Lambda^\star_\mi(x\slash \beta)$, so if $\mi$ 
satisfies (\ref{V_IC}) for this $V$, it satisfies $\IC(\beta)$.

On the other hand, suppose $\mi$ satisfies $\IC(\beta)$. Take an arbitrary 
nonempty set $V$. If the right-hand side supremum is infinite, the inequality
is obvious, so we may assume it is equal to some $s<\infty$. 
This means that for any $(v_0,\tilde{v}) \in V$ we have 
$v_0 + \Lambda_\mi(\beta\tilde{v}) \leq \log s$, that is 
$v_0 \leq \log s - \Lambda_\mi(\beta\tilde{v})$. 
Thus 
\begin{align*}
\psi_V(x) & = \sup_{v\in V} \{v_0 + \is{x}{\tilde{v}}\} \leq 
\log s + \sup_{v\in V} \{\is{x}{\tilde{v}} - \Lambda_\mi(\beta\tilde{v})\} 
\\
&\leq \log s + \sup_{\tilde{v} \in \R^n} \{\is{x}{\tilde{v}} 
- \Lambda_\mi(\beta\tilde{v})\} = \log s + \Lambda^\star_\mi(x\slash \beta),
\end{align*} 
which in turn means from $\IC(\beta)$ that the left hand side is no
larger than $s$.
\end{proof}

\begin{prop} \label{linearIC}
Let $L\colon \R^n \ra \R^k$ be a linear map and suppose  that
a probability measure $\mi$ on $\R^n$ satisfies $\IC(\beta)$. 
Then the probability measure $\mi\circ L^{-1}$ satisfies $\IC(\beta)$.
\end{prop}

\begin{proof} 
For any set $V \subset \R \times \R^k$ and any function $f\colon \R^k \ra \R$ 
put $\bar{f}(x) := f(L(x))$ and 
$\bar{V} := \{(v_0,L^\star(\tilde{v})) : (v_0,\tilde{v}) \in V\}$, where 
$L^\star$ is the Hermitian conjugate of $L$. 
Then direct calculation shows $\psi_V(L(x)) = \psi_{\bar{V}}(x)$
and $f\Box \psi_V(L(x))\leq \bar{f}\Box \psi_{\bar{V}}(x)$,
thus $$\int_{\R^k} e^{f\Box \psi_V} d(\mi\circ L^{-1}) 
\leq \int_{\R^n} e^{\bar{f}\Box \psi_{\bar{V}}} d\mi$$ 
and $$\int_{\R^k} e^{-f} d(\mi\circ L^{-1}) = \int_{\R^n} e^{-\bar{f}} d\mi$$ 
and finally 
$$
\sup_{v\in V} \Big\{e^{v_0} \int_{\R^k} e^{\beta \is{x}{\tilde{v}}} 
d(\mi\circ L^{-1})\Big\} = 
\sup_{v\in \bar{V}} \Big\{
e^{v_0} \int_{\R^n} e^{\beta \is{x}{\tilde{v}}} d\mi\Big\},
$$ 
which substituted into (\ref{V_IC}) gives the thesis.
\end{proof}

\begin{prop}
For any $x\in\er$,
$$
\frac{1}{5}\min(x^2,|x|)\leq \Lambda_{\nu}^{*}(x)\leq \min(x^2,|x|),
$$
in particular the measure $\nu$ satisfies $\IC(9)$.
\end{prop}

\begin{proof}
Direct calculation shows that $\Lambda_{\nu}(x)=-\ln(1-x^2)$ for $|x|<1$
and 
$$
\Lambda_{\nu}^\star(x)=\sqrt{1+x^2}-1-\ln\bigg(\frac{\sqrt{1+x^2}+1}{2}\bigg).
$$
Since $a/2\leq a-\ln(1+a/2)\leq a$ for $a\geq 0$, we get
$\frac{1}{2}(\sqrt{1+x^2}-1)\leq \Lambda_{\nu}^\star(x)\leq
\sqrt{1+x^2}-1$. Finally
$$
\min(x,|x|^2)\geq \sqrt{1+x^2}-1 = \frac{x^2}{\sqrt{1+x^2}+1}
\geq \frac{1}{\sqrt{2}+1}\min(|x|,x^2).
$$
The last statement follows by Theorem \ref{inf_nu}, 
since $\min((x/9)^2,|x|/9)\leq w(x)$.
\end{proof}

\subsection{Logaritmically concave product measures} 

A measure $\mu$ on $\er^n$ is {\em logarithmically concave} (log--concave for short) if
for all nonempty compact sets $A,B$ and $t\in [0,1]$,
$$
\mu(tA+(1-t)B)\geq \mu(A)^t\mu(B)^{1-t}.
$$
By Borell's theorem \cite{Bo} a measure $\mu$ on $\er^n$ with a
full--dimensional support
is logarithmically concave if and only if it is absolutely continuous
with respect to the Lebesgue measure and has a logarithmically concave
density, i.e. $d\mu(x)=e^{h(x)}dx$ for some concave function
$h\colon \er^n\rightarrow [-\infty,\infty)$.

Note that if $\mi$ is a probabilistic, symmetric and log--concave measure on $\R^n$, then both $\Lambda_\mi$ and $\Lambda^\star_\mi$ are convex and symmetric, and $\Lambda_\mi(0) = \Lambda_\mi^\star(0) = 0$.

Recall also that a probability measure $\mu$ on $\er^n$ is called 
{\em isotropic} if
$$
\int \is{u}{x}d\mu(x)=0\ \mbox{ and }
\int \is{u}{x}^2d\mu(x)=|u|^2
\ \mbox{ for all } u\in\er^n.
$$ It is easy to check that for any measure $\mi$ with a full--dimensional support there exists a linear map $L$ such that
$\mi \circ L^{-1}$ is isotropic.

The next theorem (with a different universal, but rather large constant)
may be deduced from the results of Gozlan \cite{Go}. We give the following,
relatively short proof for the sake of completeness.

\begin{thm} 
Any symmetric log-concave measure on $\R$ satisfies $\IC(48)$. 
\end{thm}

\begin{proof}
Let $\mu$ be a symmetric log--concave probability measure on $\R$, we may assume $\mi$ is isotropic by 
Proposition \ref{linearIC}. Denote the density of $\mi$ by $g(x)$ and let the tail function be $\mu[x,\infty)=e^{-h(x)}$. Let  
$$
a:=\inf\{x>0\colon g(x)\leq e^{-1}g(0)\},
$$
then $g(x)\leq e^{-x/a}g(0)$ for $x>a$ and $g(x)\geq e^{-x/a}g(0)$ for
$x\in[0,a)$. Therefore
$$
\int_{0}^{\infty}g(x)dx\leq ag(0)+\int_{a}^{\infty}g(0)e^{-x/a}dx,
$$
that is
\begin{equation}
\label{est_ag1}
1/2 \leq ag(0)(1+e^{-1}).
\end{equation}
We also have
$$
\int_{0}^ax^2g(0)e^{-x/a}dx\leq \int_{0}^{\infty}x^2g(x)dx,
$$
so
\begin{equation}
\label{est_ag2}
a^3g(0)(2-5e^{-1})\leq 1/2.
\end{equation}
From (\ref{est_ag1}) and (\ref{est_ag2}) we get in particular
$$
\frac{1}{8}\leq \sqrt{\frac{e^2(2e-5)}{4(e+1)^3}}\leq g(0).
$$
Let $T\colon \er\rightarrow\er$ be a function such that
$\nu(-\infty,x)=\mu(-\infty,Tx)$. Then $\mu=\nu\circ T^{-1}$,
$T$ is odd and concave on $[0,\infty)$. In particular,
$|Tx-Ty|\leq 2|T(x-y)|$ for all $x,y\in \er$.

Notice that $T'(0)=1/(2g(0))\leq 4$, thus by concavity of $T$, $Tx\leq 4x$ for
$x\geq 0$. Moreover for $x\geq 0$, $h(Tx)=x+\ln 2$.

Define
$$
\tilde{h}(x):=\left\{
\begin{array}{ll}
x^2 & \mbox{ for } |x|\leq 2/3
\\
\max\{4/9,h(x)\} & \mbox{ for } |x|> 2/3.
\end{array}
\right.
$$

We claim that $(\mu,\tilde{h}(\frac{\cdot}{48}))$ has property $(\tau)$.
Notice that $\tilde{h}((Tx-Ty)/48)\leq \tilde{h}(T(|x-y|)/24)$ so
by Proposition \ref{tau_trans}  it is enough to check that
\begin{equation}
\label{est_htild}
\tilde{h}\Big(\frac{Tx}{24}\Big)\leq w(x) \mbox{ for } x\geq 0,
\end{equation}
where $w(x)$ is as in Theorem \ref{inf_nu}. We have two cases.\\
i) $Tx\leq 16$, then
$$
\tilde{h}\Big(\frac{Tx}{24}\Big)=\Big(\frac{Tx}{24}\Big)^2\leq
\min\Big\{\frac{4}{9},\Big(\frac{x}{6}\Big)^2\Big\}\leq w(x). 
$$
ii) $Tx\geq 16$, then $x\geq 4$ and
\begin{align*}
\tilde{h}\Big(\frac{Tx}{24}\Big)&=
\max\Big\{\frac{4}{9},h\Big(\frac{Tx}{24}\Big)\Big\}\leq 
\max\Big\{\frac{4}{9},\frac{h(Tx)}{24}\Big\}= 
\max\Big\{\frac{4}{9},\frac{x+\ln 2}{24}\Big\}\leq 
\frac{x}{9}
\\
&\leq w(x).
\end{align*}
So (\ref{est_htild}) holds in both cases.

To conclude we need to show that $\Lambda_{\mu}^*(x)\leq \tilde{h}(x)$.
For $|x|\leq 2/3$ it follows from the more general Proposition \ref{sq_above}
below. Notice
that for any $t,x\geq 0$, $\Lambda_{\mu}(t)\geq tx+\ln\mu[x,\infty)=tx-h(x)$, 
hence
$$
\Lambda_{\mu}^*(x)=\Lambda_{\mu}^*(|x|)=
\sup_{t\geq 0}\big\{t|x|-\Lambda_{\mu}(t)\big\}
\leq h(|x|)\leq \tilde{h}(x)
$$
for $|x|>2/3$.
\end{proof}

Using Corollary \ref{product_IC} we get

\begin{cor} 
\label{SIClogprod} 
Any symmetric, log--concave product probability measure on 
$\R^n$ satisfies $\IC(48)$.
\end{cor}

We expect that in fact a more general fact holds.

\begin{conj}
\label{conj_IC}
Any symmetric log--concave probability measure satisfies $\IC(C)$ with
a uniform constant $C$.
\end{conj}

\section{Concentration inequalities.}\label{Conc_In_Section}

\subsection{$L_p$-centroid bodies and related sets}

\begin{Def}
Let $\mu$ be a probability measure on $\er^n$, for $p\geq 1$ we define the 
following sets
$$
{\cal M}_{p}(\mu):=\Big\{v\in \er^n\colon \int |\is{v}{x}|^p d\mu(x)\leq 1\Big\},
$$
$$
{\cal Z}_{p}(\mu):=({\cal M}_{p}(\mu))^{\circ}=
\Big\{x\in \er^n\colon \big|\is{v}{x}\big|^p
\leq \int \big|\is{v}{y}\big|^p d\mu(y)
\mbox{ for all } v\in \er^n\Big\}
$$
and for $p>0$ we put
$$
B_{p}(\mu):=\{v\in \er^n\colon \Lambda_{\mu}^{*}(v)\leq p\}.
$$
\end{Def}

Sets ${\cal Z}_{p}(\mu_K)$ for $p\geq 1$, when $\mu_K$ is the uniform 
disribution on the convex body $K$ are called {\em $L_p$-centroid bodies of $K$}, 
their properties were investigated in \cite{Pa}.

\begin{prop}
\label{Z_sub_B}
For any symmetric probability measure $\mu$ on $\er^n$ and $p\geq 1$,
$$
{\cal Z}_{p}(\mu)\subset 2^{1/p}eB_{p}(\mu).
$$
\end{prop}

\begin{proof}
Let us take $v\in {\cal Z}_{p}(\mu)$, we need to show that 
$\Lambda_{\mu}^{\star}(v/(2^{1/p}e))\leq p$, that is
\[
\frac{\is{u}{v}}{2^{1/p}e}-\Lambda_\mu(u)\leq p 
\mbox{ for all }
u\in \er^n.
\]
Let us fix $u\in \er^n$ with  $\int |\is{u}{x}|^pd\mu(x)=\beta^p$, then
$u/\beta\in {\cal M}_{p}(\mu)$. 
We will consider two cases.

i) $\beta\leq 2^{1/p}ep$. Then, since 
$\Lambda_{\mu}(u)\geq \int \langle u,x\rangle d\mu(x) =0$,
$$
\frac{\is{u}{v}}{2^{1/p}e}-\Lambda_{\mu}(u)\leq
\frac{\beta}{2^{1/p}e}\is{\frac{u}{\beta}}{v}\leq p\cdot 1.
$$

ii) $\beta> 2^{1/p}ep$. We have
\begin{align*}
\int e^{\is{u}{x}}d\mu(x)&\geq
\int \big|e^{\is{u}{x}/p}\big|^pI_{\{\is{u}{x}\geq 0\}}d\mu(x)
\geq \int \Big|\frac{\is{u}{x}}{p}\Big|^pI_{\{\is{u}{x}\geq 0\}}d\mu(x)
\\
&\geq \frac{1}{2}\int \Big|\frac{\is{u}{x}}{p}\Big|^p d\mu(x),
\end{align*}
thus
$$
\int e^{2^{1/p}ep\is{u}{x}/\beta}d\mu(x)\geq
\frac{1}{2}\int \Big|\frac{2^{1/p}e\is{u}{x}}{\beta}\Big|^{p}d\mu(x)
= e^p.
$$
Hence $\Lambda_{\mu}(2^{1/p}epu/\beta)\geq p$ and
$\Lambda_{\mu}(u)\geq \frac{\beta}{2^{1/p}ep}
\Lambda_{\mu}(2^{1/p}epu/\beta)
\geq \frac{\beta}{2^{1/p}e}$. Therefore
$$
\frac{\is{u}{v}}{2^{1/p}e}-\Lambda_{\mu}(u)\leq
\frac{\beta}{2^{1/p}e}\is{\frac{u}{\beta}}{v}-\frac{\beta}{2^{1/p}e}
\leq 0.
$$
\end{proof}

\begin{prop}
\label{sq_above}
If $\mu$ is a symmetric, isotropic probability measure on $\er^n$, then
$\min\{1,\Lambda_{\mu}^*(u)\}\leq |u|^2$ for all $u$, in particular
$$
\sqrt{p}B_2^n\subset B_p(\mu) \mbox{ for } p\in (0,1).
$$
\end{prop}

\begin{proof}
Using the symmetry and isotropicity of $\mu$, we get
$$
\int e^{\is{u}{x}}d\mu(x)=
1+\sum_{k=1}^{\infty}\frac{1}{(2k)!}\int \is{u}{x}^{2k}d\mu(x)
\geq 1+\sum_{k=1}^{\infty}\frac{|u|^{2k}}{(2k)!}=\cosh (|u|).
$$
Hence for $|u|<1$,
$$
\Lambda_{\mu}^{*}(u)\leq {\cal L}(\ln \cosh)(|u|)=
\frac{1}{2}\Big[(1+|u|)\ln(1+|u|)+(1-|u|)\ln (1-|u|)\Big]\leq |u|^2,
$$
where to get the last inequality we used $\ln(1+x)\leq x$ for
$x>-1$.
\end{proof}

\subsection{$\alpha$-regular measures.}

To establish inlusions opposite to those in the previous subsection, we 
introduce the following property: 

\begin{Def}
We say that a measure $\mu$ on $\er^n$ is $\alpha$-regular if for
any $p\geq q\geq 2$ and $v\in\er^n$,
$$
\Big(\int |\is{v}{x}|^p d\mu(x)\Big)^{1/p}\leq
\alpha\frac{p}{q}\Big(\int |\is{v}{x}|^q d\mu(x)\Big)^{1/q}.
$$
\end{Def}

\begin{prop}
\label{B_sub_Z}
If $\mu$ is $\alpha$-regular for some $\alpha\geq 1$, 
then for any $p\geq 2$,
$$
B_{p}(\mu)\subset 4e\alpha {\cal Z}_{p}(\mu).
$$
\end{prop}

\begin{proof}
First we will show that 
\begin{equation}
\label{claim_1}
u\in {\cal M}_{p}(\mu)\ \Rightarrow\
\Lambda_{\mu}\Big(\frac{pu}{2e\alpha}\Big)\leq p.
\end{equation}
Indeed if we fix $u\in {\cal M}_{p}(\mu)$ and put
$\tilde{u}:=\frac{pu}{2e\alpha}$, then
$$
\Big(\int |\is{\tilde{u}}{x}|^kd\mu(x)\Big)^{1/k}=
\frac{p}{2e\alpha}\Big(\int |\is{u}{x}|^kd\mu(x)\Big)^{1/k}\leq
\left\{
\begin{array}{ll}
\frac{p}{2e\alpha} & k\leq p
\\
\frac{k}{2e} & k>p.
\end{array}
\right.
$$
Hence
\begin{align*}
\int e^{\is{\tilde{u}}{x}}d\mu(x)&\leq
\int e^{|\is{\tilde{u}}{x}|}d\mu(x)=
\sum_{k=0}^{\infty}\frac{1}{k!}\int|\is{\tilde{u}}{x}|^kd\mu(x)
\\
&\leq 
\sum_{k\leq p}\frac{1}{k!}\Big|\frac{p}{2e\alpha}\Big|^k+
\sum_{k>p}\frac{1}{k!}\Big|\frac{k}{2e}\Big|^k
\leq e^{\frac{p}{2e\alpha}}+1\leq e^{p}
\end{align*}
and (\ref{claim_1}) follows.

Take any $v\notin 4e\alpha {\cal Z}_{p}(\mu)$, then we may find
$u\in {\cal M}_{p}(\mu)$ such that $\is{v}{u}> 4e\alpha$ and 
obtain
$$
\Lambda_{\mu}^{*}(v)\geq \is{v}{\frac{pu}{2e\alpha}}-
\Lambda_{\mu}\Big(\frac{pu}{2e\alpha}\Big)> \frac{p}{2e\alpha}4e\alpha-p
=p.
$$
\end{proof}

\begin{prop}\label{sq_below}
If $\mu$ is symmetric, isotropic $\alpha$-regular for some 
$\alpha\geq 1$, then 
$$
\Lambda_\mu^*(u)\geq \min\Big\{\frac{|u|}{2\alpha e},\frac{|u|^2}{2\alpha^2e^2}\Big\},
$$
in particular
$$
B_p(\mu)\subset \max\{2\alpha ep ,\alpha e\sqrt{2 p}\} B_2^n
\mbox{ for all } p>0.
$$
\end{prop}

\begin{proof}
We have by the symmetry, isotropicity and regularity of $\mu$,
\begin{align*}
\int e^{\is{u}{x}}d\mu(x)&=
\sum_{k=0}^{\infty}\frac{1}{(2k)!}\int\is{u}{x}^{2k}d\mu(x)
\leq 1+\frac{|u|^2}{2}+
\sum_{k=2}^{\infty}\frac{(\alpha k|u|)^{2k}}{(2k)!}
\\
&\leq 1+\frac{|u|^2}{2}+
\sum_{k=2}^{\infty}\Big(\frac{\alpha e|u|}{2}\Big)^{2k}.
\end{align*}
Hence if $\alpha e|u|\leq 1$,
$$
\int e^{\is{u}{x}}d\mu(x)\leq
1+\frac{|u|^2}{2}+\frac{4}{3}\Big(\frac{\alpha e|u|}{2}\Big)^{4}\leq
1 + \alpha^2 e^2 |u|^2 + \frac{(\alpha e |u|)^4}{2} \leq
e^{\alpha^2 e^2 |u|^2 \slash 2}
$$
so $\Lambda_\mu(u)\leq \alpha^2 e^2|u|^2/2$ for
$\alpha e|u|\leq 1$. Thus 
$\Lambda_\mu^*(u)\geq \min\{\frac{|u|}{2\alpha e},\frac{|u|^2}{2\alpha^2e^2}\}$ for 
all $u$. 
\end{proof}

\begin{remark}
\label{growth_Z}
We always have for $p\geq q$, ${\cal M}_p(\mu)\subset {\cal M}_q(\mu)$
and  ${\cal Z}_q(\mu)\subset {\cal Z}_p(\mu)$. If the measure 
$\mu$ is $\alpha$-regular, then
 ${\cal M}_q(\mu)\subset \frac{\alpha p}{q}{\cal M}_p(\mu)$ and
${\cal Z}_p(\mu)\subset \frac{\alpha p}{q}{\cal Z}_q(\mu)$ for
$p\geq q\geq 2$. Moreover for any symmetric measure $\mu$, 
$\Lambda_{\mu}^*(0)=0$, hence by the convexity of $\Lambda_{\mu}^*$,
$B_{q}(\mu)\subset B_{p}(\mu)\subset \frac{p}{q}B_{q}(\mu)$
for all $p\geq q>0$.
\end{remark}

\begin{prop}\label{LogToReg}
Symmetric log--concave measures are 1-regular.
\end{prop}

\begin{proof}
If $X$ is distributed according to a symmetric, log--concave
measure $\mu$ and $u\in \er^n$, then the random variable 
$S=\langle u, X\rangle$ has a log--concave symmetric distribution
on the real line. We need to show that 
$(\Ex|S|^p)^{1/p}\leq \frac{p}{q}(\Ex|S|^q)^{1/q}$ for $p\geq q\geq 2$.
The proof of Remark 5 in \cite{KLO} shows
that 
\[
(\Ex|S|^p)^{1/p}\leq
\frac{(\Gamma(p+1))^{1/p}}{(\Gamma(q+1))^{1/q}}(\Ex|S|^q)^{1/q},
\]
so it is enough to show that the function
$f(x):=\frac{1}{x}(\Gamma(x+1))^{1/x}$ is nonincreasing on $[2,\infty)$.
Binet's form of the Stirling formula (cf.\ \cite[Theorem 1.6.3]{AAR})  gives
\[
\Gamma(x+1)=x\Gamma(x)=\sqrt{2\pi}x^{x+1/2}e^{-x+\mu(x)},
\]
where $\mu(x)=\int_0^\infty{\rm arctg}(t/x)(e^{2\pi t}-1)^{-1}dt$ is 
decreasing function. Thus
\[
\ln f(x)= \frac{\mu(x)}{x}+\frac{\ln (2\pi x)}{2x} - 1
\]
is indeed nonincreasing on $[2,\infty)$.
\end{proof}

Let us introduce the following notion:
\begin{Def}
\label{CIDef}
We say that a measure $\mu$ satifies the {\em concentration inequality with
constant $\beta$} ($\CI(\beta)$ in short) if
\begin{equation}
\label{in_CI}
\forall_{p\geq 2}\forall_{A\in {\cal B}(\er^n)}\
\mu(A)\geq \frac{1}{2}\ \Rightarrow\
1-\mu(A+\beta{\cal Z}_{p}(\mu))\leq e^{-p}(1-\mu(A)).
\end{equation}
\end{Def}

This definition is motivated by the following Corollary:

\begin{cor}
\label{IC_and_CI}
Let $\mu$ be an $\alpha$--regular symmetric and isotropic probability measure 
with $\alpha \geq 1$. Then \\
i) If $\mu$ satisfies $\IC(\beta)$, then $\mu$ satisfies  $\CI(8e\alpha\beta)$,
\\
ii) If $\mu$ satisfies $\CI(\beta)$ and additionally satisfies Cheeger's 
inequality (\ref{cheeger}) with constant $1\slash\gamma$, then $\mu$ 
satisfies $\IC(36\min\{6e\beta, \gamma\})$.
\end{cor}

\begin{proof}
By Remark \ref{growth_Z}, Proposition \ref{ICconc} and the definition
of $B_p(\mu)$ we have
\[
\mu(A+2\beta B_{p}(\mu))\geq \mu(A+\beta B_{2p}(\mu))\geq 
1-e^{-p}(1-\mu(A)),
\]
so the first part of the statement immediately follows by Proposition \ref{B_sub_Z}.

On the other hand, if $\mu$ satisfies $\CI(\beta)$, then by Remark \ref{growth_Z} and Proposition \ref{Z_sub_B} we have for $\mu(A) \geq 1\slash 2$ and $p \geq 1$
\begin{align*} e^{-p} (1 - \mi(A)) & > e^{-2p} (1 - \mi(A)) \geq 1 - \mi(A + \beta {\cal Z}_{2p}(\mi)) 
\\ & \geq 1 - \mi(A + e2^{1\slash 2p} \beta B_{2p}(\mi)) \geq 1 - \mi(A + 3e \beta B_{p}(\mi)).\end{align*}
By Proposition \ref{equiv} this implies property (\ref{concsm}). Additionally $\Lambda^\star_{\mi}$ is convex, symmetric and $\Lambda^\star_\mi(0) = 0$. Finally, from Proposition \ref{sq_above} we have $\min\{1,\Lambda^\star_\mi(u)\} \leq |u|^2$. Thus, from Proposition \ref{reduction} we get the second part of the statement.

\end{proof}

By Proposition \ref{equiv} in the definition \ref{CIDef} we could use the equivalent condition $\mu(A+\beta{\cal Z}_{p}(\mu))\geq \min\{e^{p}\mu(A),1/2\}$.
The next proposition shows that for log-concave measures these conditions
are satisfied for large $p$ and for small sets.

\begin{prop}
\label{small_sets}
Let $\mu$ be a symmetric log-concave probability measure on $\er^n$ and
$c\in (0,1]$. Then
$$
\mu\Big(A+\frac{40}{c}{\cal Z}_{p}(\mu)\Big)\geq \frac{1}{2}\min\{e^{p}\mu(A),1\}
$$
for $p\geq cn$ or $\mu(A)\leq e^{-cn}$.
\end{prop}

\begin{proof}
Using a standard volumetric estimate for any $r > 0$ we may choose 
$S\subset {\cal M}_{r}(\mu)$ 
with $\# S\leq 5^n$  such that 
${\cal M}_{r}(\mu)\subset \bigcup_{u\in S}(u+\frac{1}{2}{\cal M}_{r}(\mu))$.
Then for $t>0$,
$$
x\notin t{\cal Z}_{r}(\mu)\ \Rightarrow\ \max_{u\in S}\is{u}{x}\geq t/2
$$
and by the Chebyshev inequality,
$$
\mu\big(\er^n\setminus t{\cal Z}_{r}(\mu)\big)\leq
\sum_{u\in S}\mu\Big\{x\colon \is{u}{x}\geq \frac{t}{2}\Big\}\leq
\sum_{u\in S}\Big(\frac{2}{t}\Big)^r\int \is{u}{x}_+^r d\mu
\leq \frac{1}{2}5^{n}\Big(\frac{2}{t}\Big)^r.
$$

Let $\mu(A)=e^{-q}$, we will consider two cases.

i) $p\geq \max\{q,cn\}$. Then by Remark \ref{growth_Z},
$$
\mu(30c^{-1}{\cal Z}_{p}(\mu))>\mu(30{\cal Z}_{\max\{p,n\}})\geq
 1-\frac{1}{2}e^{-\max\{p,n\}}\geq 1-\mu(A),
$$
so $A\cap 30c^{-1}{\cal Z}_{p}(\mu)\neq \emptyset$, 
hence $0\in A+30c^{-1}{\cal Z}_{p}(\mu)$ and
$$
\mu(A+40c^{-1}{\cal Z}_{p}(\mu))\geq 
\mu(10c^{-1}{\cal Z}_{p}(\mu))\geq 1/2.
$$

ii) $q\geq \max\{p,cn\}$. Let $\tilde{q}:=\max\{q,n\}$
$$
\tilde{A}:=A\cap 30c^{-1}{\cal Z}_{q}(\mu),
$$
we have as in i),
$\mu(30c^{-1}{\cal Z}_{q}(\mu))>1-e^{-\tilde{q}}/2$, thus
$\mu(\tilde{A})\geq \mu(A)/2$. Moreover,
$$
\Big(1-\frac{p}{q}\Big)\tilde{A} \subset
A-\frac{p}{q}30c^{-1}{\cal Z}_{q}(\mu)\subset
A+30c^{-1}{\cal Z}_{p}(\mu)
$$
and
\begin{align*}
\mu\big(A+40c^{-1}{\cal Z}_{p}(\mu)\big)&\geq
\mu\Big(\Big(1-\frac{p}{q}\Big)\tilde{A}
+\frac{p}{q}10c^{-1}{\cal Z}_{q}(\mu)\Big)
\\
&\geq \mu\Big(\Big(1-\frac{p}{q}\Big)\tilde{A}+\frac{p}{q}10
{\cal Z}_{\tilde{q}}(\mu)\Big)
\geq \mu(\tilde{A})^{1-\frac{p}{q}}\mu(10{\cal Z}_{\tilde{q}})^{\frac{p}{q}}
\\
&\geq \Big(\frac{1}{2}\mu(A)\Big)^{1-\frac{p}{q}}
\Big(\frac{1}{2}\Big)^{\frac{p}{q}}
\geq \frac{1}{2}\mu(A)\mu(A)^{-\frac{p}{q}}=\frac{1}{2}e^{-p}\mu(A).
\end{align*}

\end{proof}

The previous facts motivate the following.

\begin{conj}
\label{conj_CI}
Any symmetric log--concave probability measure satisfies $\CI(C)$ for some
universal constant $C$.
\end{conj}

Proposition \ref{IC_and_CI} shows that Conjecture \ref{conj_IC} implies
Conjecture \ref{conj_CI}. Both hypotheses would be equivalent provided
that the following conjecture of Kannan, Lov\'asz and Simonovits holds.

\begin{conj}[Kannan--Lov\'asz--Simonovits \cite{KLS}]
\label{conj_KLS}
There exists an absolute constant $C$ such that any symmetric isotropic 
log--concave probability measure satisfies Cheeger's
inequality with constant $1/C$.
\end{conj}

\subsection{Comparison of weak and strong moments}

\begin{prop}
\label{weakstr}
Suppose that a probability measure $\mu$ on $\er^n$
is $\alpha$-regular and satisfies
$\CI(\beta)$.
 Then for any norm $\|\cdot\|$ on $\er^n$ and $p\geq 2$,
$$
\Big(\int \big|\|x\| -\Med_{\mu}(\|x\|)\big|^pd\mu\Big)^{1/p}\leq
2\alpha\beta
\sup_{\|u\|_*\leq 1}\Big(\int |\is{u}{x}|^pd\mu\Big)^{1/p},
$$
where $\|\cdot\|_*$ denotes the norm dual to $\|\cdot\|$.
\end{prop}

\begin{proof}
For $p\geq 2$ we define
$$
m_p:=\sup_{\|u\|_*\leq 1}\Big(\int |\is{u}{x}|^pd\mu\Big)^{1/p}.
$$
Let $M:=\Med_{\mu}(\|x\|)$, $A:=\{x\colon \|x\|\leq M\}$
and $\tilde{A}:=\{x\colon \|x\|\geq M\}$.
Then $\mu(A),\mu(\tilde{A})\geq 1/2$ so by $\CI(\beta)$ and 
Remark \ref{growth_Z},
$$
\forall_{t\geq p}\ 
1-\mu\Big(A+\beta\frac{\alpha t}{p}{\cal Z}_{p}(\mu)\Big)\leq 
\frac{1}{2}e^{-t},\ 
1-\mu\Big(\tilde{A}+\beta\frac{\alpha t}{p}{\cal Z}_{p}(\mu)\Big)\leq 
\frac{1}{2}e^{-t}.
$$
Let $y\in {\cal Z}_{p}$, then there exists $u\in \er^n$ with
$\|u\|_*\leq 1$ such that
$$
\|y\|=\is{u}{y}\leq \Big(\int |\is{u}{x}|^pd\mu(x)\Big)^{1/p}\leq m_p,
$$
hence $\|x\|\leq M+tm_p$ for $x\in A+t{\cal Z}_{p}(\mu)$. 
Thus for $t\geq p$, 
$$
\mu\Big\{x\colon \|x\|\geq M+\frac{\alpha\beta t}{p}m_p\Big\}\leq
1-\mu\Big(A+\beta\frac{\alpha t}{p}{\cal Z}_{p}(\mu)\Big)\leq 
\frac{1}{2}e^{-t}.
$$
In a similar way we show $\|x\|\geq M-tm_p$ for 
$x\in \tilde{A}+t{\cal Z}_{p}(\mu)$ and
$\mu\{x\colon \|x\|\leq M-\alpha\beta t m_p/p\}\leq e^{-t}/2$, therefore
$$
\mu\Big\{x\colon \big|\|x\|-M\big|\geq 
\frac{\alpha\beta t}{p}m_p\Big\}\leq e^{-t}
\mbox{ for } t\geq p.
$$
Now integrating by parts,
\begin{align*}
\Big(\int |\|x\|-M|^p&d\mu\Big)^{1/p}
\\
&\leq
\frac{\alpha\beta m_p}{p}
\Big[p+\Big(p\int_{p}^{\infty} t^{p-1}
\mu\Big\{x\colon \big|\|x\|-M\big|\geq \frac{\alpha\beta t}{p}m_p
\Big\}dt\Big)^{1/p}\Big]
\\
&\leq \frac{\alpha\beta m_p}{p}
\Big[p+\Big(p\int_{p}^{\infty} t^{p-1}e^{-t}dt\Big)^{1/p}\Big]
\\
&\leq \alpha\beta m_p\Big(1+\frac{\Gamma(p+1)^{1/p}}{p}\Big)\leq
2\alpha\beta m_{p}.
\end{align*}
\end{proof}

\begin{remark} Under the assumptions of Proposition \ref{weakstr} by the 
triangle inequality we get for $\gamma=4\alpha\beta$,
\begin{equation}
\label{wsm_g}
\forall_{p\geq q\geq 2}\
\Big(\int \Big|\|x\| -\Big(\int \|x\|^q d\mu \Big)^{1/q}\Big|^p
d\mu\Big)^{1/p}\leq
\gamma
\sup_{\|u\|^*\leq 1}\Big(\int |\is{u}{x}|^pd\mu\Big)^{1/p}.
\end{equation}
\end{remark}

This motivates the following definition.

\begin{Def}
We say that a probability measure $\mu$ on $\er^n$ has {\em comparable
weak and strong moments} with the constant $\gamma$ ($\CWSM(\gamma)$ in short)
if (\ref{wsm_g}) holds for any norm $\|\cdot\|$ on $\er^n$.
\end{Def}

\begin{conj}
\label{conj_cwsm}
Every symmetric log--concave probability on $\er^n$ measure satisfies $\CWSM(C)$.
\end{conj}

\begin{prop}\label{conseq}
Let $\mu$ be an isotropic, probability measure on $\er^n$ satisfying
$\CWSM(\gamma)$. Then\\
i) $\int|\|x\|_2-\sqrt{n}|^2d\mu(x)\leq \gamma^2$,
\\
ii) if $\mu$ is also $\alpha$--regular, then for all $p>2$,
$$
\Big(\int \|x\|_2^pd\mu\Big)^{1/p}\leq \sqrt{n}+\frac{\gamma\alpha}{2}p.
$$
\end{prop}

\begin{proof}
Notice that $\int \|x\|_2^2d\mu=n$ and $\|u\|_2^*=\|u\|_2$. Hence
i) follows directly from (\ref{wsm_g}) with $p=q=2$. 
Moreover (\ref{wsm_g}) with $q=2$ implies
$$
\Big(\int \|x\|_2^pd\mu\Big)^{1/p}\leq \sqrt{n}+
\sup_{\|u\|_2\leq 1}\Big(\int |\is{u}{x}|^pd\mu\Big)^{1/p}
\leq \sqrt{n}+\frac{\gamma\alpha}{2}p
$$
by the $\alpha$-regularity and isotropicity of $\mu$.
\end{proof}

\begin{remark}
Property i) plays the crucial role
in the Klartag proof of the central limit theorem for convex bodies \cite{Kl}.
Paouris \cite{Pa} showed that moments of the Euclidean norm for symmetric
isotropic log-concave measures are bounded by $C(p+\sqrt{n})$. Thus
Conjecture \ref{conj_cwsm} would imply both Klartag CLT (with the optimal
speed of convergence) and Paouris concentration.
\end{remark}

We conclude this section with the estimate that shows comparison 
of weak and strong moments for any probability measure and $p>n/C$. 

\begin{prop}
For any $p>0$ we have 
\begin{align*}
\Big(\int\big|\|x\|-\Med_{\mu}(\|x\|)\big|^{p}d\mu\Big)^{1/p}
&\leq \Big(\int \|x\|^pd\mu\Big)^{1/p} 
\\
&\leq 2\cdot 5^{n/p}\sup_{\|u\|_*\leq 1}\Big(\int |\is{u}{x}|^pd\mu\Big)^{1/p}.
\end{align*}
%%Similar estimate holds for $(\int \|x\|^pd\mu)^{1/p}$.
\end{prop}

\begin{proof}
As in the proof of Proposition \ref{small_sets} we can find
$u_1,\ldots,u_N$ with $\|u_i\|_{*}\leq 1$, $N\leq 5^n$ such that
$\|x\|\leq 2\max_{i\leq N}\is{u_i}{x}$ for all $x$. Then
$$
\int \|x\|^pd\mu\leq 2^p\int\sum_{i\leq N}|\is{u_i}{x}|^p d\mu\leq
2^p5^n\sup_{\|u\|_*\leq 1}\int|\is{u_i}{x}|^pd\mu.
$$
Moreover 
$$
\int_{\{\|x\|\geq M\}} (\|x\|-M)^p d\mu(x)\leq
\int_{\{\|x\|\geq M\}} (\|x\|^p-M^{p})d\mu(x)\leq
\int \|x\|^pd\mu(x)-\frac{1}{2}M^p
$$
and
$$
\int_{\{\|x\|<M\}}(M-\|x\|)^pd\mu(x)\leq
M^p\mu\{x\colon \|x\|<M\}\leq \frac{1}{2}M^p.
$$ 
\end{proof}

\section{Modified Talagrand concentration for exponential measure}

In this section we show that for a set lying far from the origin
Talagrand's two level concentration for the exponential measure may be somewhat 
improved, namely (for sufficiently large $t$) it is enough to enlarge
the set by $tB_1^n$ instead of $tB_1^n+\sqrt{t}B_2^n$.

\begin{lem} 
\label{wpychanieB1}
If $u\geq t >0$ then for any $i \in \{1,\ldots,n\}$ we have
\[
\big|\big(A+tB_1^n\big)\cap nB_1^n\cap \big\{x\colon |x_i|\geq u-t\big\}\big|\geq
e^{t/2}\big|A\cap nB_1^n\cap\big\{x\colon |x_i|\geq u\big\}\big|. 
\] 
\end{lem}

\begin{proof}
Obviously we may assume that $i=1$ and $u\leq n$. Let 
$A_1:=A\cap nB_1^n\cap\{x\colon x_1\geq u\}$ and
$B:= \{x \in B_1^n : x_1 \geq \sum_{i \geq 2} |x_i|\}$. From the definition 
of $B$ and $A_1$ we have $A_1 - tB \subset n B_1^n$. On the other hand 
$B = \{x : |x_1 - 1/2| + \sum_{i \geq 2} |x_i| \leq 1/2\}$, 
so $|B| = 2^{-n} |B_1^n| = (2r_{1,n})^{-n}$.
Thus 
$$
|(A_1 + t B_1^n) \cap n B_1^n| \geq |(A_1 - t B) \cap n B_1^n| 
= |A_1 - t B|
$$. 

Now let us take 
$$ 
s:= \frac{2 |A_1|^{1\slash n}r_{1,n}}{t + 2 |A_1|^{1\slash n}r_{1,n}}. 
$$
Then we easily check that $|tB/(1-s)|=|A_1/s|$. Since 
$A_1\subset \{x\in nB_1^n\colon x_1\geq t\}$ we get
$|A_1|^{1/n}\leq (n-t)\slash r_{1,n}$ and $s\leq 2(n-t)/(2n-t)$.
Now we can use the Brunn-Minkowski inequality to get 
\begin{align*}
|A_1 - tB| =& 
\Big|s \frac{A_1}{s} + (1-s) \frac{-t}{1-s}B\Big| \geq 
\Big|\frac{A_1}{s}\Big|^{s} \Big|\frac{-t}{1-s}B\Big|^{1-s} = 
\Big|\frac{A_1}{s}\Big| =  s^{-n}|A_1| 
\\ 
\geq &\Big(\frac{2n-t}{2n-2t}\Big)^n|A_1| =
\Big(\frac{1}{1-\frac{t}{2n-t}}\Big)^n|A_1|\geq
e^{\frac{tn}{2n-t}} |A_1| \geq e^{t/2}|A_1|.
\end{align*}

Notice that $A_1+tB_1^n\subset \{x\colon x_1\geq u-t\}$, so we 
obtain 
$$
\big|\big(A+tB_1^n)\cap nB_1^n\cap \big\{x\colon x_1\geq u-t\big\}\big|\geq
e^{t/2}\big|A\cap nB_1^n\cap\big\{x\colon x_1\geq u\big\}\big|, 
$$ 
in the same way we show
$$
\big|\big(A+tB_1^n\big)\cap nB_1^n\cap \big\{x\colon x_1\leq -u+t\big\}\big|\geq
e^{t/2}\big|A\cap nB_1^n\cap\big\{x\colon x_1\leq -u\big\}\big|. 
$$ 
\end{proof}

\begin{remark} 
A similar result (although with a constant multiplicative factor) can be obtained using the same technique and more calculations for $n^{1/p}B_p^n$ 
instead of $nB_1^n$ for $p \in [1,2]$.
\end{remark}

\begin{lem} \label{wpychanieWykladniczy}
If $u\geq t >0$ then for any $i \in \{1,\ldots,n\}$ we have
$$
\nu^{n}\big(\big(A+tB_1^n\big)\cap \big\{x\colon |x_i|\geq u-t\big\}\big)\geq
e^{t/2}\nu^{n}\big(A\cap\big\{x\colon |x_i|\geq u\big\}\big).
$$ 
\end{lem}

\begin{proof} Take an arbitrary $k \in \N$. Let $P: \R^{n+k} \ra \R^n$ 
be the projection onto first $n$ coordinates. Let $\rho_k$ be the uniform
probability measure on $(n+k)B_1^{n+k}$, and $\tilde{\nu}_k$ the measure defined by
$\tilde{\nu}_k(A) = \rho_k(P^{-1}(A))$. 
Take an arbitrary set $A \subset \R^n$. Notice that for any set 
$C \subset \R^n$ we have 
$$
C \cap \{x \colon |x_i| \geq s\} = 
P \big( P^{-1}(C) \cap \{x \colon |x_i| \geq s\}\big)
$$
and also
$P^{-1}(A)+B_1^{n+k} \subset P^{-1} (A+B_1^n)$. From Lemma 
\ref{wpychanieB1} we have 
$$
\rho_k\big(\big(P^{-1}(A) + tB_1^{n+k}\big) 
\cap \big\{x \colon |x_i| \geq u - t\big\}\big) 
\geq e^{t/2}\rho_k\big(P^{-1}(A) \cap \big\{x \colon |x_i| \geq u\big\}\big),
$$
and thus 
$$
\tilde{\nu}_k\big(\big(A + tB_1^n\big) 
\cap \big\{x \colon |x_i| \geq u - t\big\}\big) 
\geq e^{t/2} \tilde{\nu}_k\big(A \cap \big\{x \colon |x_i| \geq u\big\}\big).
$$

When $k \ra \infty$, we have $\tilde{\nu}_k(C) \ra \nu^n(C)$ for any set 
$C \in {\cal B}(R^n)$. Thus by going to the limit we get the assertion.
\end{proof}

\begin{prop} 
\label{CalDrugiejNormy}
For any $t > 0$ and any  $n \in \N$ we have 
$$
\int_{A + t B_1^n} |x|^2 d\nu^n(x) \geq e^{t/2} 
\int_A (|x|-t\sqrt{n})_+^2 d\nu^n(x).
$$
\end{prop}

\begin{proof}
Let $A_t = A + t B_1^n$. By Lemma \ref{wpychanieWykladniczy} we get for 
any $s \geq 0$ and any $i$:
$$
\int_{A_t} I_{\{|x_i| \geq s\}} d\nu^n(x) 
\geq e^{t/2}  \int_A I_{\{|x_i| \geq s + t\}}d\nu^n(x).
$$
Thus
\begin{align*} 
\int_{A_t} x_i^2 d\nu^n(x) & 
= \int_{A_t} \int_0^\infty 2s I_{\{|x_i| \geq s\}} ds\ d\nu^n(x) 
= \int_0^\infty 2s \int_{A_t} I_{\{|x_i| \geq s\}} d\nu^n(x) ds 
\\ 
& \geq e^{t/2} \int_0^\infty 2s \int_A I_{\{|x_i| \geq s + t\}} d\nu^n(x) ds 
\\
&= e^{t/2} \int_A \int_0^\infty 2s I_{\{|x_i| \geq s + t\}} ds\ d\nu^n(x) 
 = e^{t/2} \int_A \big(|x_i| - t\big)_+^2 d\nu^n(x) .
\end{align*}
To get the assertion it is enough to take the sum over all $i$ and notice
that
the function $f(y):=(\sqrt{y}-t)_{+}^2$ is convex on $[0,\infty)$, hence
$$
\sum_{i=1}^{n}(|x_i| - t)_+^2=\sum_{i=1}^{n}f(x_i^2)\geq
nf\Big(\frac{1}{n}\sum_{i=1}^{n}x_i^2\Big)=(|x|-t\sqrt{n})_+^2
$$
\end{proof}

\begin{lem} \label{pojedynczePchniecie}
Suppose that 
$A\subset\{x\in \er^n\colon |x|\geq 5t\sqrt{n}\}$.
Then 
$$
\nu^n(A + t B_1^n) \geq \frac{1}{8}e^{t/2} \nu^n(A).
$$
\end{lem}

\begin{proof}
Let 
$$
A_k := A \cap \{x \colon 5t \sqrt{n} + 2t (k-1) \leq |x| < 5t \sqrt{n} + 2t k\},
\quad k=1,2,\ldots.
$$
Then $A_k+tB_1^n \subset 
\{x\colon 5t\sqrt{n}+t(2k-3) \leq |x| < 5t\sqrt{n}+t(2k+1)\}$,
hence
$$
 \nu^n(A+tB_1^n) \geq \frac{1}{2} \sum_{k\geq 1} \nu(A_k+tB_1^n).
$$

From Proposition \ref{CalDrugiejNormy} applied for $A_k$ we have
\begin{align*}
\big(5t\sqrt{n}+t(2k&+1)\big)^2  \nu^n(A_k+tB_1^n)
\geq \int_{A_k + t B_1^n} |x|^2 d\nu^n(x) 
\\
 &\geq  e^{t/2} \int_{A_k} \big(|x|-t\sqrt{n})_+^2d\nu^n(x)
 \geq e^{t/2} \big(4t\sqrt{n}+2t(k-1)\big)^2\nu^n(A_k).
\end{align*}
Thus
$$
\nu^n(A_k + t B_1^n) \geq 
\Big(\frac{4t\sqrt{n}+2t(k-1)}{5t\sqrt{n}+t(2k+1)}\Big)^2e^{t/2}
\nu_n(A_k)\geq 
\frac{1}{4}e^{t/2}\nu^n(A_k).
$$
and
$$
\nu^n(A+tB_1^n) \geq 
\frac{1}{2} \sum_{k\geq 1}\frac{1}{4}e^{t/2}\nu^n(A_k)
=\frac{1}{8}e^{t/2}\nu^n(A).
$$
\end{proof}

\begin{thm} \label{PushAndPop} 
For any $A \in  \BB({\R^n})$ and any $t \geq 10$,
either
$$
\nu^n\big((A + t B_1^n) \cap 50 \sqrt{n} B_2^n\big) \geq  \frac{1}{2}\nu^n(A)
$$
or
\begin{equation}
\label{l1_enl}
\nu^n(A + t B_1^n) \geq e^{t \slash 10} \nu^n(A).
\end{equation}
In particular $(\ref{l1_enl})$ holds if
$A\cap(50\sqrt{n}B_2^n+tB_1^n)=\emptyset$.
\end{thm}

\begin{proof}
Let $A_k$ denote $A + 10k B_1^n$ for $k=0,1,\ldots$. If for any $0\leq k\leq t/10$ 
we have 
$\nu^n(A_k \cap 50 \sqrt{n} B_2^n) \geq \nu^n(A) \slash 2$, the thesis is proved. Thus we assume otherwise. Let $A_k' := A_k \setminus 50 \sqrt{n} B_2^n$. 
From Lemma \ref{pojedynczePchniecie} we have 
$$
\nu^n(A_{k+1}) \geq \nu^n(A_k' + 10B_1^n) \geq 
\frac{1}{8}e^{5}\nu^n(A_k') \geq \frac{1}{16}e^5\nu^n(A_k)
\geq e^{2}\nu^n(A_k).
$$
By a simple induction we get $\nu^n(A_k) \geq e^{2k} \nu^n(A)$
for any $k\leq t/10$. Thus we 
get 
$$
\nu^n(A + t B_1^n) \geq \nu^n\big(A_{\lfloor t/10\rfloor}\big) 
\geq e^{2\lfloor t/10\rfloor} \nu^n(A) \geq e^{t\slash 10} \nu(A).
$$
\end{proof}

\section{Uniform measure on $B_p^n$}

In this section we will prove the infimum convolution property $\IC(C)$ for $B_p^n$
balls. Recall that $\npn$ is a product measure, while $\mpn$ denotes the uniform 
measure on $\cpn B_p^n$.
We have 
$$
r_{p,n}^{-n}=|B_p^n| = \frac{2^n \Gamma(1 + 1\slash p)^n}{\Gamma(1 + n \slash p)} 
\sim \frac{(2\Gamma(1+1\slash p))^n (ep)^{n\slash p}}
{n^{n\slash p}(\sqrt{n\slash p} + 1)},
$$
where the last part follows from Stirling's formula. 
Thus $\cpn\sim n^{1/p}$. 

For $\npn$ we have $\IC(48)$ by Corollary \ref{SIClogprod}. Let us first
try to understand what sort of concentration this implies, that is, how does the 
function $\Lambda^\star$ behave for $\npn$.

\begin{prop} 
\label{ksztaltLambdaStarNu} 
For any $p \geq 1$ and $t\in\R$ we have 
$$
B_t(\np1) \sim \{x : f_p(|x|) \leq t\},
\mbox{ and } 
\Lambda^\star_\np1 (t/C) \leq f_p(|t|) \leq \Lambda^\star_{\np1}(Ct),
$$ 
where $f_p(t) = t^2$ for $t < 1$ and $f_p(t) = t^p$ for $t \geq 1$.
\end{prop}

\begin{proof}We shall use the facts proved in Section \ref{Conc_In_Section} 
to approximate $B_t(\np1)$. Note that $\np1$ is log-concave (as its density 
is log-concave) and symmetric. It is 1--regular from Proposition \ref{LogToReg}. 
Also 
$$
\sigma_p^2:=\int_\R x^2 d\nu_{p}(x) = \frac{1}{2\gamma_p}\int_\R x^2 e^{-|x|^p}dx 
=\frac{\Gamma(1+\frac{3}{p})}{3\Gamma(1 + \frac{1}{p})} \sim 1
$$ 
for $p \in [1,\infty)$. The measure $\tilde{\nu}_p$ with the density 
$\sigma_p d\nu_{p}({\sigma_px})$ is isotropic, hence Propositions 
\ref{sq_above} and \ref{sq_below} yield $B_t(\tilde{\nu}_p) \sim \sqrt{t} B_2^1 
= [-\sqrt{t},\sqrt{t}]$ for $t \leq 1$. 
Thus, as $B_{t}(\np1)=\sigma_p B_t(\tilde{\nu}_p)$, we get 
$B_t(\np1) \sim [-\sqrt{t},\sqrt{t}]$ for $t \leq 1$.

For $t \geq 1$ we have 
\begin{align*} 
\mathcal{M}_t(\np1) & 
= \Big\{u\in\R\colon \frac{1}{2\gamma_p}\int_\R |u|^t |x|^t e^{-|x|^p} dx
\leq 1\Big\}
\\ & 
= \Bigg\{u\in\R\colon |u|\leq\sqrt[t]{\frac{(t+1)\Gamma(1 + \frac{1}{p})}
{\Gamma(1 + \frac{t+1}{p})}}\Bigg\} 
\sim \{u\in\R\colon |u| \leq t^{-1\slash p}\}.
\end{align*}
Thus $Z_t(\np1) \sim [-t^{1\slash p},t^{1\slash p}]$ for $|t| \geq 1$, 
so by Propositions \ref{Z_sub_B} and \ref{B_sub_Z},  
$B_t(\np1) \sim [-t^{1\slash p}, t^{1\slash p}]$. Hence, for all $t \geq 0$ 
we have $\{x\colon f_p(|x|) \leq t\} \sim \{x : \Lambda^\star_\np1(x) \leq t\}$,
so $\Lambda^\star_\np1 (t/C) \leq f_p(t) \leq \Lambda^\star_{\np1}(Ct)$. 
As $\Lambda^\star_\np1$ is symmetric, the proof is finished.
\end{proof}

\begin{cor} 
\label{KsztaltBpNu}
For any $t > 0$ and $n \in \N$ we have 
$$
B_t(\npn) \sim
\left\{\begin{array}{ll} 
\sqrt{t} B_2^n + t^{1\slash p}B_p^n& \mbox {for } p\in [1,2]
\\
\sqrt{t}B_2^n \cap t^{1\slash p}B_p^n& \mbox {for } p\geq 2.
\end{array}\right.
$$
\end{cor}

\begin{proof}
By Proposition \ref{ksztaltLambdaStarNu},
$$
B_t(\npn)=\{x\in \er^n\colon \sum \Lambda_{\np1}^{\star}(x_i)\leq t\}
\sim \{x\in \er^n\colon \sum f_{p}(|x_i|)\leq t\}.
$$
Simple calculations show that 
$\{x\in \er^n\colon \sum f_{p}(|x_i|)\leq t\}\sim t^{1/2}B_2^n+t^{1/p}B_p^n$
for $p\in [1,2]$ and
$\{x\in \er^n\colon \sum f_{p}(|x_i|)\leq t\}\sim t^{1/2}B_2^n\cap t^{1/p}B_p^n$
for $p\geq 2$.
\end{proof}

\begin{prop} 
\label{ksztaltLambdaStarMi} 
For any $t\in [0,n]$, $p \geq 1$ and $n \in \N$ we have $B_t(\mpn) \sim B_t(\npn)$. 
\end{prop}

\begin{proof} 
For $t < 1$ we use Propositions \ref{sq_above} and \ref{sq_below}. Both $\mpn$ and
$\npn$ are symmetric, log--concave measures, and both can be rescaled as in Proposition
\ref{ksztaltLambdaStarNu} to be isotropic, thus $B_t(\mpn) \sim \sqrt{t} B_2^n \sim B_t(\npn)$.

Lemma 6 from \cite{BGMN} gives (after rescaling by $\cpn$),
\begin{equation}
\label{comp_mom}
\Big(\int |\is{a}{x}|^t d\mpn(x) \Big)^{1\slash t} \sim 
\frac{\cpn}{(\max\{n,t\})^{1\slash p}}\Big(\int|\is{a}{x}|^t d\npn(x)\Big)
^{1\slash t}
\end{equation}
for any $p,t\geq 1$ and $a\in \er^n$.
Note that as $\cpn \sim n^{1\slash p}$, this simply means 
the equivalence of $t$-th moments of $\mu_{p,n}$ and $\nu_{p,n}$ for $t\in [0,n]$. 
Thus $\mathcal{M}_t(\mu_{p,n}) \sim \mathcal{M}_t(\nu_{p,n})$ for $t \leq n$ 
and therefore $B_t(\mu_{p,n}) \sim B_t(\nu_{p,n})$.
\end{proof}

\begin{remark}
It is not hard to verify that $B_t(\mpn) \sim \cpn B_p^n$ for
$t\geq n$.
\end{remark}

\subsection{Transports of measure}
\label{Transporty}

We are now going to investigate two transports of measure. They will combine to 
transport a measure with known concentration properties ($\nu^n$ or $\nu_2^n$, 
that is the exponential or Gaussian measure) to the uniform measure $\mi_{p,n}$. 
We will investigate the contractive properties of these transports with respect to 
various norms. Our motivation is the following:

\begin{remark} 
\label{normComparisons} 
Let $U : \R^n \ra \R^n$ be a map such that
$$
\|U(x) - U(y)\|_p^p \geq \delta\|x-y\|_q^q
\mbox{ for all } x\in\R^n, y\in A.
$$ 
Then 
$$
U\big(A + t^{1\slash q} B_q^n\big) 
\supset U\big(\R^n\big) \cap \big(U(A) + \delta^{1\slash p} t^{1/p} B_p^n\big).
$$ 
Analogously if
$$
\|U(x) - U(y)\|_p^p \leq \delta\|x-y\|_q^q \mbox{ for all } x\in\R^n, y \in A
$$ 
then 
$$
U\big(A + t^{1\slash q}B_q^n\big) \subset 
U(A) + \delta^{1\slash p} t^{1\slash p}B_p^n.
$$
\end{remark}

\begin{proof}
Let us prove the first statement, the second proof is almost identical. 
Suppose $U(x) \in U(A) + \delta^{1\slash p} t^{1\slash p} B_p^n$. Then there 
exists 
$y \in A$ such that $\|U(x) - U(y)\|_p^p \leq \delta t.$ 
From the assumption we have $t \geq \|x - y \|_q^q$, which means 
$x \in A + t^{1\slash q} B_q^n$, and $U(x) \in U(A + t^{1\slash q} B_q^n)$.
\end{proof}

The first transport we introduce is the radial transport $T_{p,n}$ which 
transforms the product measure $\nu_p^n$  onto $\mi_{p,n}$ -- 
the uniform measure on $\cpn B_p^n$. 
We will show this transport is Lipschitz with respect to the $\ell_p$ norm 
and Lipschitz on a large set with respect to the $\ell_2$ norm for $p \leq 2$.

\begin{Def}
For $p\in [1,\infty)$ and $n\in \N$
let $f_{p,n} : [0,\infty) \ra [0,\infty)$ be given by the equation 
\begin{equation}
\label{fnp_1}
\int_0^s e^{-r^p} r^{n-1} dr = (2\gamma_p)^n \int_0^{f_{p,n}(s)} r^{n-1} dr
\end{equation}
and $T_{p,n}(x):= x f_{p,n}(\|x\|_p) \slash \|x\|_p$ for $x\in \er^n$.
\end{Def}

Let us first show the following simple estimate.

\begin{lem}
\label{temp_gamma}
For any $q>0$ and $0\leq u\leq q/2$,
\[
q\int_{0}^{u}e^{-t}t^{q-1}dt\leq e^{-u}u^{q}\Big(1+2\frac{u}{q}\Big).
\] 
\end{lem}

\begin{proof}
Let 
$$
f(u):=e^{-u}u^{q}\Big(1+2\frac{u}{q}\Big)-q\int_{0}^{u}e^{-t}t^{q-1}dt.
$$
Then $f(0)=0$ and $f'(u)=e^{-u}u^q(1-2u/q+2/q)\geq 0$ for $0\leq u\leq q/2$. 
\end{proof}

Now we are ready to state the basic properties of $T_{p,n}$.

\begin{prop} 
\label{transportBpProperties} 
i) The map $T_{p,n}$  
 transports the probability measure $\nu_p^n$ onto the measure $\mi_{p,n}$.
\\
ii) For all $t > 0$ we have $e^{-t^p \slash n} t \leq 2\gamma_p f_{p,n}(t) \leq t$ 
and $f_{p,n}'(t) \leq (2\gamma_p)^{-1}\leq 1$.
\\
iii) For any $t>0$, $0\leq f_{p,n}(t)/t-f_{p,n}'(t)\leq \min\{1,2pt^{p}/n\}$.
\\
iv) The function $t \mapsto f_{p,n}(t) \slash t$ is decreasing on $(0,\infty)$ 
and for any $s,t > 0$, 
$$
|t^{-1} f_{p,n}(t) - s^{-1} f_{p,n}(s)| \leq (st)^{-1} |s-t| f_{p,n}(s\wedge t) 
\leq \frac{|s-t|}{\max\{s,t\}}.
$$
\end{prop}

\begin{proof}
The definition of $T_{p,n}$ directly implies i). 
Differentiation of (\ref{fnp_1}) gives
\begin{equation}
\label{fpn_2}
e^{-s^p} s^{n-1} = (2\gamma_p)^n f_{p,n}^{n-1}(s) f_{p,n}'(s).
\end{equation}

By (\ref{fnp_1}),
$$
e^{-t^p} t^n %%= e^{-t^p} n \int_0^t r^{n-1}dr 
\leq n \int_0^t e^{-r^p} r^{n-1} dr = (2\gamma_p)^n f_{p,n}^n(t) 
\leq n \int_0^t r^{n-1} dr = t^n,
$$ 
which, when the $n$-th root is taken, give the first part of ii).

For the second part of ii) we use (\ref{fpn_2}) and the
estimate above to get
\begin{align*}
f_{p,n}'(s) &= e^{-s^p} (2\gamma_p)^{-n}\Big(\frac{s}{f_{p,n}(s)}\Big)^{n-1}  
\leq 
e^{-s^p} (2\gamma_p)^{-n} \big(e^{s^p\slash n} 2\gamma_p\big)^{n-1} 
\\
&= e^{-s^p \slash n} (2\gamma_p)^{-1} \leq (2\gamma_p)^{-1}\leq 1.
\end{align*}

To show iii) first notice that by (\ref{fpn_2}) and ii),
$$
\frac{tf_{p,n}'(t)}{f_{p,n}(t)} = \Big(\frac{t}{f_{p,n}(t)}\Big)^n e^{-t^p}
(2\gamma_p)^{-n} 
\leq \Big(e^{t^p \slash n} 2\gamma_p\Big)^ne^{-t^p} (2\gamma_p)^{-n} = 1,
$$ 
thus $f_{p,n}(t)/t-f_{p,n}'(t)\geq 0$. Moreover by ii),
$f_{p,n}(t)/t-f_{p,n}'(t)\leq f_{p,n}(t)/t\leq 1$, so we may assume that
$2pt^p/n\leq 1$. By (\ref{fnp_1}) and Lemma \ref{temp_gamma} we obtain
$$
(2\gamma_p)^nf_{p,n}^n(t)=\frac{n}{p}\int_{0}^{t^p}e^{-u}u^{n/p-1}du
\leq e^{-t^p}t^n\Big(1+2\frac{pt^p}{n}\Big).
$$
Thus using again (\ref{fpn_2}) and part ii) we get
$$
\frac{f_{p,n}(t)}{t}-f_{p,n}'(t)=\frac{f_{p,n}(t)}{t}
\bigg(1-\frac{e^{-t^p}t^n}{(2\gamma_p)^nf_{p,n}^n(t)}\bigg)
\leq 
1-\Big(1+2\frac{pt^p}{n}\Big)^{-1}\leq
\frac{2pt^p}{n}.
$$

By iii) we get $(f_{p,n}(t) \slash t)' \leq 0$, 
which proves the first part of iv). For the second part 
suppose that $s > t > 0$, then
$$ 
0\leq \frac{f_{p,n}(t)}{t} - \frac{f_{p,n}(s)}{s} \leq 
\frac{f_{p,n}(t)}{t} - \frac{f_{p,n}(t)}{s} = 
\frac{s-t}{st} f_{p,n}(t) \leq \frac{s-t}{s}.
$$
\end{proof}

The next Proposition may be also deduced (with different constant) from
the more general fact proved in \cite{MS}.

\begin{prop} 
\label{kontrakcjaLp} 
For any $x,y\in \er^n$ we have 
$\|T_{p,n}x - T_{p,n}y\|_p \leq 2\|x-y\|_p$.
\end{prop}

\begin{proof} 
Assume $s:=\|x\|_p\geq t:=\|y\|_p$, we apply
Proposition \ref{transportBpProperties} and get
\begin{align*}
\|T_{p,n}x - T_{p,n}y\|_p 
& = \Big(\sum_i \big|(T_{p,n}x)_i - (T_{p,n}y)_i\big|^p\Big)^{1/p}
\\
&= \Big(\sum_i \Big| \frac{f_{p,n}(t)}{t} (x_i - y_i) + 
\Big(\frac{f_{p,n}(s)}{s} - \frac{f_{p,n}(t)}{t}\Big) x_i\Big|^p\Big)^{1/p}
\\ 
& \leq 
\Big(\sum_i \Big(|x_i - y_i| + \frac{|s-t|}{s}|x_i|\Big)^p\Big)^{1/p} 
\\ 
& \leq 
 \Big(\sum_i |x_i - y_i|^p\Big)^{1/p} +
\frac{|s-t|}{s}\Big(\sum_i |x_i|^p\Big)^{1/p} 
\\ 
& = \|x - y\|_p + 
\frac{\big|\|x\|_p - \|y\|_p\big|}{\|x\|_p} \|x\|_p\leq
%%&  \Big(\|x-y\|_p + \|x-y\|_p\Big) =
 2 \|x-y\|_p.
\end{align*}
\end{proof}

\begin{prop} 
\label{kontrakcjaL2} 
Let $u\geq 0$, $p\in [1,2]$ and 
$x\in \R^n$ be such that $\|x\|_2 n^{-1\slash 2} \leq u \|x\|_p n^{-1\slash p}$,
then 
$$
\|T_{p,n}x - T_{p,n}y\|_p \leq (1 + u) \|x-y\|_p
\mbox{ for all } y\in \er^n.
$$
\end{prop}

\begin{proof} 
Let $s=\|x\|_p$ and $t=\|y\|_p$, we use Proposition \ref{transportBpProperties}
as in the proof of Proposition \ref{kontrakcjaLp},
and the H\"older inequality,
\begin{align*}
\|T_{p,n}x - T_{p,n}y\|_2 
%%= \Big(\sum_i |(T_{p,n}x)_i - (T_{p,n}y)_i|^2\Big)^{1/2}
&\leq   \Big(\sum_i\Big| |x_i - y_i|+|x_i| \frac{|s-t|}{s}\Big|^2
\Big)^{1/2} 
\\ 
& 
\leq  \|x - y\|_2 + \frac{|s-t|}{s} \|x\|_2 
\leq \|x-y\|_2 + \frac{\|x-y\|_p}{\|x\|_p} \|x\|_2
\\
& \leq  \|x-y\|_2 + 
\frac{\|x\|_2}{\|x\|_p} n^{\frac{1}{p} - \frac{1}{2}} \|x-y\|_2 
\leq (1+u) \|x-y\|_2.
\end{align*}
\end{proof}

The second transport we will use is a simple product transport which transports the 
measure $\nu_p^n$ onto $\nu_q^n$. We shall be particularly interested in the cases 
$p = 1$ and $p = 2$, but most of the results can be stated in the more general 
setting.

\begin{Def}
For $1 \leq p, q< \infty$ we define the map $w_{p,q}:\R\ra\R$ by 
\begin{equation}
\label{defw} \frac{1}{\gamma_p} \int_x^\infty e^{-t^p} dt = 
\frac{1}{\gamma_q}\int_{w_{p,q}(x)}^\infty e^{-t^q} dt.
\end{equation}
By $v_p$ we denote $w_{p,1}$. We also define $W_{p,q}^n : \R^n \ra \R^n$ by 
$$
W_{p,q}^n(x_1,x_2,\ldots,x_n) = (w_{p,q}(x_1),w_{p,q}(x_2),\ldots,w_{p,q}(x_n)).
$$
\end{Def}

Note that $w_{p,q}^{-1} = w_{q,p}$ and $(W_{p,q}^n)^{-1} = W_{q,p}^n$. 
Differentiating equality (\ref{defw}) we get 
\begin{equation}
\label{difw}
w_{p,q}'(x) = \frac{\gamma_q}{\gamma_p} e^{-x^p + w_{p,q}^q(x)}.
\end{equation}

We will prove that $w_{p,q}$ behaves very much like $x^{p\slash q}$ for large $x$, 
and is more or less linear for small $x$. We begin with the bound for $q=1$.

\begin{lem}
\label{est_vp}
For $p\geq 1$ we have\\
i) $v_p(x)\geq x^p+\ln(p\gamma_p x^{p-1})$ and $v_p'(x)\geq px^{p-1}$ for
$x\geq 0$,\\
ii) $v_p(x)\leq e+ x^p+\ln(p\gamma_p x^{p-1})$ and $v_p'(x)\leq e^e px^{p-1}$ 
for $x\geq 1$,\\
iii) $|v_p(x)-v_p(y)|\geq 2^{1-p}|x-y|^p$.
\end{lem}

\begin{proof}
Note that $\gamma_1 = 1$. We have for $x\geq 0$,
\begin{equation} \label{IntegralpUpper}
e^{-v_p(x)}=\frac{1}{\gamma_p}\int_x^\infty e^{-t^p} dt 
\leq \frac{1}{p\gamma_px^{p-1}} \int_x^\infty p t^{p-1} e^{-t^p} dt=
\frac{e^{-x^p}}{p\gamma_px^{p-1}}
\end{equation}
and for $x\geq 1$, since 
$(1 + r\slash p)^p \leq e^r \leq 1 + er$ for $r \in [0,1]$, we get
$$
e^{-v_p(x)}\geq
\frac{1}{\gamma_p}\int_x^{x+x^{1-p}/p} e^{-t^p} \geq 
\frac{1}{p\gamma_px^{p-1}} e^{-(x + x^{1-p}/p)^p} 
\geq e^{-e}\frac{e^{-x^p}}{p\gamma_p x^{p-1}}.
$$
Notice that by (\ref{difw}), $v_p'(x)=e^{-x^p+v_p(x)}/\gamma_p$, hence
we may estimate $v_p'$ using the just derived bounds on $v_p$.

The lower bound on $v_p'$ yields $|v_p(x)-v_p(y)|\geq |x-y|^p$ for $x,y\geq 0$. The same 
estimate holds for $x,y\leq 0$, since $v_p$ is odd. Finally for $x\geq 0\geq y$ we
have
$$
|v_p(x)-v_p(y)|=|v_p(x)|+|v_p(y)|\geq |x|^p+|y|^p\geq 2^{1-p}|x-y|^p.
$$
\end{proof}

\begin{lem}
\label{est_wpq}
i) For $p\geq q\geq 1$,
 $|w_{p,q}(x)|\geq |x|^{p/q}$ and
 $w_{p,q}'(x)\geq\frac{\gamma_q}{\gamma_p}\geq\frac{1}{2}$. 
\\
ii) For $p\geq 2$, $w_{p,2}'(x)\geq \frac{1}{8} \sqrt{p} |x|^{p\slash 2 - 1}.$
\end{lem}

\begin{proof}
Since the function $w_{p,q}$ is odd, we may and will assume that $x\geq 0$.

i) We have by the monotonicity of $u^{p/q-1}$ on $[0,\infty)$,
\begin{align*}
\frac{1}{\gamma_p} \int_x^\infty e^{-t^p} dt 
&= \frac{1}{\gamma_q} \int_{w_{p,q}(x)}^\infty e^{-t^q} dt 
= \frac{\int_{w_{p,q}(x)}^\infty e^{-t^q} dt}{\int_0^\infty e^{-t^q} dt} 
= \frac{\int_{w_{p,q}(x)^{q\slash p}}^\infty u^{p \slash q - 1} e^{-u^p} du}
{\int_0^\infty u^{p \slash q - 1} e^{-u^p} du} 
\\
&\geq 
\frac{\int_{w_{p,q}(x)^{q\slash p}}^\infty e^{-u^p} du}{\int_0^\infty e^{-u^p} du} 
= \frac{1}{\gamma_p} \int_{w_{p,q}(x)^{q\slash p}}^\infty e^{-u^p} du,
\end{align*}
thus $w_{p,q}(x)^{q \slash p} \geq x$ and $w_{p,q}(x) \geq x^{p\slash q}$.
Formula (\ref{difw}) gives $w_{p,q}'(x) \geq \gamma_q/\gamma_p \geq 1/2$.

ii) 
We begin by the following  Gaussian tail estimate for $z > 0$: 
\begin{equation}
\label{Integral2Lower}
\int_z^\infty e^{-t^2} dt \geq \frac{1}{2\sqrt{z^2+1}} e^{-z^2}.
\end{equation} 
We have equality when $z \ra \infty$, and direct calculation shows the derivative 
of the left--hand--side is no larger than the derivative of the right--hand--side.

Let $\kappa:=4\sqrt{\pi}$, we will now show that for all $x>0$ and $p\geq 2$,
\begin{equation}
\label{est_by_up}
w_{p,2}(x)\geq u_{p}(x) := 
\max\Big\{\frac{\sqrt{\pi}}{2}x,
\sqrt{\big(x^p+\ln(\sqrt{p}x^{p/2-1}/\kappa)\big)_+}\Big\}.
\end{equation}
Suppose on the contrary that $w_{p,2}(x) < u_p(x)$ for some $p\geq 2$ and $x>0$.
Note that by i) we have $w_{p,2}' \geq \gamma_2/\gamma_p \geq \gamma_2 = \sqrt{\pi}/2$. 
Thus $u_p(x)$ is equal to the second part of the maximum. This in particular 
implies
that $x\geq 2/3$, since for $x < 2/3$ we have 
$$
x^p+\ln(\sqrt{p}x^{p/2-1}/\kappa)
\leq \frac{4}{9}+\Big(\frac{p}{2}-1\Big)\ln\frac{2}{3}+
\frac{\sqrt{p}}{\kappa}-1\leq 0.
$$
Therefore $u_p(x)\geq \sqrt{\pi}x/2\geq 1/\sqrt{3}$. Now by (\ref{IntegralpUpper}),
(\ref{defw}) and (\ref{Integral2Lower}),
\begin{align*} 
\sqrt{\pi}\frac{1}{px^{p-1}}e^{-x^p} 
&\geq \frac{\gamma_2}{\gamma_p} \frac{1}{px^{p-1}} e^{-x^p} 
\geq \frac{\gamma_2}{\gamma_p} \int_x^\infty e^{-t^p}dt 
= \int_{w_{p,2}(x)}^\infty e^{-t^2} dt 
\\
&> \int_{u_p(x)}^\infty e^{-t^2} dt 
\geq \frac{1}{2\sqrt{u_p^2(x)+1}}e^{-u_p^2(x)}
\geq \frac{1}{4u_p(x)}e^{-u_p^2(x)}
\\
&
=\frac{1}{4u_p(x)}e^{-(x^p+\ln(\sqrt{p}x^{p/2-1}/\kappa))}
=\frac{\sqrt{\pi}}{\sqrt{p}u_p(x)}x^{1-p/2}e^{-x^p}.
\end{align*}
After simplifying this gives $u_p(x) > \sqrt{p}x^{p/2}$.
Hence
$$
px^p<u_p^2(x)=
x^p+\frac{1}{2}\ln(px^{p})+\ln\frac{1}{\kappa x}\leq
\frac{p}{2}x^p+\frac{1}{2}px^p=px^p,
$$ 
which is impossible. This condratiction shows that (\ref{est_by_up}) holds.

Thus we have $w_{p,2}(x) \geq u_p(x)$ and by (\ref{difw}) we obtain
$$
w_{p,2}'(x) \geq \frac{\gamma_2}{\gamma_p}e^{-x^p + u_p^2(x)} \geq 
\frac{\sqrt{\pi}}{2}\frac{1}{\kappa} \sqrt{p} x^{p/2-1}=
\frac{1}{8}\sqrt{p} x^{p/2-1}.
$$
\end{proof}

\begin{remark}
By taking $u_p(x) = \max\{\sqrt{\pi}x/2,
\sqrt{(x^p + \ln (px^{p/2-1}/(\kappa\ln p)))_+}\}$ for sufficiently large $\kappa$
and estimating carefully one may arrive at the bound 
$w_{p,2}'(x) \geq C^{-1} p x^{p/2-1}/\ln p$. One cannot, however, receive a bound 
of the order of $p x^{p\slash 2 - 1}$.
\end{remark}

\begin{prop}
\label{Wproperties}
For $p \geq q \geq 1$ we have\\
i) $\nu_p^n(W_{q,p}^n(A)) = \nu_q^n(A)$ for $A\in \BB(\R^n)$,
\\
ii)  $|w_{q,p}(x)-w_{q,p}(y)|\leq 2|x-y|$ for $x\in\R$,
\\
iii)  for $x,y \in \R^n$ and $r \geq 1$, 
$$
\|W_{q,p}^n(x) - W_{q,p}^n(y)\|_r\leq 2\|x - y\|_r,
$$
\\
iv) for $x,y \in \R$,
$$
\big|w_{1,p}(x) - w_{1,p}(y)\big| \leq 2\min(|x-y|,|x-y|^{1/p})\leq
2|x-y|^{1/q},
$$
v) $\|W_{1,p}^n(x)-W_{1,p}^n(y)\|_{q}^q \leq 2^q\|x-y\|_1$ for $x,y \in \R^n$.
\end{prop}

\begin{proof} Property i) follows from the definition
of $w_{q,p}$ and $W_{q,p}^n$. Since $w_{q,p}=w_{p,q}^{-1}$ we get
ii) by Lemma \ref{est_wpq} i). Property iii) is a direct consequence of ii).

By Lemma \ref{est_vp} iii), 
$$
|w_{1,p}(x)-w_{1,p}(y)|=|v_{p}^{-1}(x)-v_p^{-1}(y)|\leq 2^{1-1/p}|x-y|^{1/p}.
$$
The above inequality together with ii) gives iv) and iv) yields v). 
\end{proof}

Now we define a transport from the exponential measure $\nu^n$ to $\mpn$ for 
$p \geq 2$:
\begin{Def} 
For  $n\in \N$ and $2\leq p<\infty$ we define the map 
$S_{p,n}\colon \R^n \ra \R^n$ by $S_{p,n}(x):= T_{p,n}(W_{1,p}^n(x))$.
\end{Def}

This transport satisfies the following bound:

\begin{prop} 
\label{kontr_spn} 
We have $\|S_{p,n}(x)-S_{p,n}y\|_2\leq 4\|x-y\|_2$ for
all $x,y\in\R^n$ and $p\geq 2$.
\end{prop}

\begin{proof}
It is enough to show that 
$\|DS_{p,n}(x)\| \leq 4$,
where $DS_{p,n}$ is the derivative matrix, and the norm is the operator norm 
from $\ell_2^n$ into $\ell_2^n$.

Let $s = \|W_{1,p}^n(x)\|_p$. By direct calculation we get
\begin{equation}
\label{der_Spn}
\frac{(\partial S_{p,n})_j}{\partial x_i}(x) =
\frac{\delta_{ij} f_{p,n}(s) w_{1,p}'(x_i)}{s} +
\alpha(s)w_{1,p}(x_j)\beta(x_i)
\end{equation}
where
$$
\alpha(s):=s^{-p-1}
\big(s f_{p,n}'(s) - f_{p,n}(s)\big)
\mbox{ and }
\beta(t):=
|w_{1,p}(t)|^{p-1} 
{\rm sgn}(w_{1,p}(t)) w_{1,p}'(t).
$$
Thus we can bound
$$
\|DS_{p,n}(x)\| \leq \frac{f_{p,n}(s)}{s}\max_i |w_{1,p}'(x_i)|
+|\alpha(s)|\big\|W_{1,p}^n(x)\big\|_2 \Big(\sum_{i=1}^n\beta^{2}(x_i)\Big)^{1/2}.
$$

Since $w_{1,p} = w_{p,1}^{-1}$, Proposition \ref{est_wpq} i) implies 
$|w_{1,p}'(x_j)| \leq 2$, while by Proposition \ref{transportBpProperties} we have
$f_{p,n}(s)/s \leq 1$. Thus the first summand can be bounded by $2$.

For the second summand note that by Proposition \ref{transportBpProperties} iii),
\begin{equation}
\label{est_alpha}
|\alpha(s)|=s^{-p}\Big|f_{p,n}'(s)-\frac{f_{p,n}(s)}{s}\Big|\leq
s^{-p}\min\Big\{1,\frac{2ps^p}{n}\Big\}.
\end{equation}
Moreover, $\|W_{1,p}^n(x)\|_{2}\leq n^{1/2-1/p}s$ by the H\"older inequality and
$$
|\beta(t)|=|w_{1,p}(t)|^{p-1}|w_{1,p}'(t)|=\frac{|w_{1,p}(t)|^{p-1}}{v_p'(w_{1,p}(t))}
\leq \frac{1}{p}.
$$
by Lemma \ref{est_vp}. Thus
\begin{align*}
\|DS_{p,n}(x)\|&
\leq 2+s^{-p}\min\Big\{1,\frac{2ps^p}{n}\Big\}n^{1/2-1/p}s\frac{n^{1/2}}{p}
\\
&\leq 2+2sn^{-1/p}\min\{ns^{-p},1\}
\leq 4.
\end{align*}
\end{proof}

\begin{prop}
\label{kontr2_spn}
For any $y,z \in \R^n$ and $p\geq 2$ we have 
$$
\|S_{p,n}(y) - S_{p,n}(z)\|_2 \leq \|W_{1,p}^n(y) - W_{1,p}^n(z)\|_2 + 
2n^{-1/2}\|y-z\|_1.
$$
\end{prop}

\begin{proof}
Let $u_i(t) = (y_1,y_2,\ldots,y_{i-1},t,z_{i+1},z_{i+2},\ldots,z_n)$ for 
$i = 1,\ldots,n$.
Note that $u_i(y_i) = u_{i+1}(z_{i+1})$, $u_1(z_1) = z$ and $u_n(y_n) = y$, hence 
$$
S_{p,n}(z) - S_{p,n}(y) = 
\sum_{i=1}^{n} \big(S_{p,n}(u_i(z_i)) - S_{p,n}(u_i(y_i))\big).
$$

Let $s_i(t):=\|w_{1,p}(u_i(t))\|_p$. By vector--valued integration 
and (\ref{der_Spn}) we get 
$$
S_{p,n}\big(u_i(z_i)\big) - S_{p,n}\big(u_i(y_i)\big) = 
\int_{y_i}^{z_i} \frac{\partial S_{p,n}}{\partial x_i}(u_i(t))dt 
= a_i+b_i,
$$
where
$$
a_i:=\int_{y_i}^{z_i}\frac{f_{p,n}(s_i(t))}{s_i(t)} w_{1,p}'(t) e_idt
$$
and
$$
b_i:=\int_{y_i}^{z_i}\alpha(s_i(t))\beta(t)W_{1,p}^{n}(u_i(t))dt.
$$
As in the proof of Proposition \ref{kontr_spn} we show that
$$
\big\|\alpha(s_i(t))\beta(t)W_{1,p}^{n}(u_i(t))\big\|_{2}\leq
2n^{-1/2}s_i(t)n^{-1/p}\min\{ns_i(t)^{-p},1\}\leq 2n^{-1/2},
$$
thus
$$
\Big\|\sum_{i=1}^{n}b_i\Big\|_2\leq
\sum_{i=1}^{n}\|b_i\|_2\leq 2n^{-1/2}\sum_{i=1}^{n}|y_i-z_i|
=2n^{-1/2}\|y-z\|_1.
$$

To deal with the sum of $a_i$'s we notice that, since
$f_{p,n}(s)/s\leq 1$ and $w_{1,p}'(x)\geq 0$, 
\begin{align*} 
\Big|\Big\langle \sum_j a_j,e_i\Big\rangle\Big| &= 
|\is{a_i}{e_i}| = 
\Big|\int_{y_i}^{z_i}\frac{f_{p,n}(s_i(t))}{s_i(t)} w_{1,p}'(t)dt\Big| 
\\
&\leq \Big|\int_{y_i}^{z_i}w_{1,p}'(t)dt\Big| =
\big|w_{1,p}(z_i) - w_{1,p}(y_i)\big|.
\end{align*} 
Thus 
$$
\|\sum_i a_i\|_2 \leq \|\sum_i \big(w_{1,p}(z_i) - w_{1,p}(y_i)\big)e_i\|_2 
= \|W_{1,p}^n(z) - W_{1,p}^n(y)\|_2.
$$
\end{proof}

\begin{cor}
\label{KontrakcjaDuzep}
If $x-y\in tB_1^n+t^{1/2}B_{2}^n$ for some $t>0$, then for all $p\geq 2$, 
$S_{p,n}x-S_{p,n}y\in 10(t^{1/2}B_2^n\cap t^{1/p}B_p^n)$.
\end{cor}

\begin{proof}
Let us fix $x,y$ with $x-y\in tB_1^n+t^{1/2}B_{2}^n$. By Proposition
\ref{Wproperties} iv),
\begin{align*}
\|W_{1,p}^n(x)-W_{1,p}^n(y)\|_p^p&=
\sum_{i}|w_{1,p}(x_i)-w_{1,p}(y_i)|^p
\\
&\leq
2^p\sum_{i}\min(|x_i-y_i|^p,|x_i-y_i|)
\\
&\leq 2^p\sum_{i}\min(|x_i-y_i|^2,|x_i-y_i|)\leq 2^{p+2}t.
\end{align*}
Thus by Proposition \ref{kontrakcjaLp}, 
$$
\|S_{p,n}x-S_{p,n}y\|_p\leq 2\|W_{1,p}^nx-W_{1,p}^ny\|_p\leq 8t^{1/p}.
$$

By H\"older's inequality $\|S_{p,n}x-S_{p,n}y\|_2\leq n^{1/2-1/p}\|S_{p,n}x-S_{p,n}y\|_p
\leq 8t^{1/2}$ for $t\geq n$.

Assume now that $t\leq n$.
Let $z$ be such that $x-z\in t^{1/2}B_2^n$ and $z-y\in tB_1^n$. Then
$S_{p,n}x-S_{p,n}z\in 4t^{1/2}B_2^n$ by Proposition \ref{kontr_spn} and
$\|W_{1,p}^n z-W_{1,p}^ny\|_2\leq 2\sqrt{t}$ by Proposition \ref{Wproperties} v).
Thus by Proposition \ref{kontr2_spn},
$$
\|S_{p,n}x-S_{p,n}z\|_2\leq 4t^{1/2}+2n^{-1/2}t\leq 6t^{1/2}.
$$
Hence $S_{p,n}x-S_{p,n}y\in 10t^{1/2}B_2^n$.
\end{proof}

The last function we define transports the Gaussian measure $\nu_2^n$ to $\mpn$ 
for $p \geq 2$.
\begin{Def} 
For  $n\in \N$ and $2\leq p<\infty$ we define the map 
$\tilde{S}_{p,n}\colon \R^n \ra \R^n$ by $\tilde{S}_{p,n}(x):= T_{p,n}(W_{2,p}^n(x))$.
\end{Def}

\begin{prop} 
\label{kontr_tildespn} 
We have $\|\tilde{S}_{p,n}(x)-\tilde{S}_{p,n}y\|_2\leq 
18\|x-y\|_2$ for
all $x,y\in\R^n$ and $p\geq 2$.
\end{prop}

\begin{proof}
We argue in a similar way as in the proof of Proposition  \ref{kontr_spn}.
We need to show that 
$\|D\tilde{S}_{p,n}(x)\| \leq 18$. Direct calculation gives 
\begin{equation}
\label{der_tildespn}
\frac{(\partial \tilde{S}_{p,n})_j}{\partial x_i}(x) =
\frac{\delta_{ij} f_{p,n}(\tilde{s}) w_{2,p}'(x_i)}{\tilde{s}} +
\alpha(\tilde{s})w_{2,p}(x_j)\tilde{\beta}(x_i)
\end{equation}
where $\tilde{s}=\|W_{2,p}^n(x)\|_p$,
$$
\alpha(s):=s^{-p-1}
\big(s f_{p,n}'(s) - f_{p,n}(s)\big)
\mbox{ and }
\tilde{\beta}(t):=
|w_{2,p}(t)|^{p-1} 
{\rm sgn}(w_{2,p}(t)) w_{2,p}'(t).
$$
Thus we can bound
\begin{equation}
\label{est_dtspn}
\|D\tilde{S}_{p,n}(x)\| \leq 
\frac{f_{p,n}(\tilde{s})}{\tilde{s}}\max_i |w_{2,p}'(x_i)|
+|\alpha(\tilde{s})|\big\|W_{2,p}^n(x)\big\|_2
 \Big(\sum_{i=1}^n\tilde{\beta}^{2}(x_i)\Big)^{1/2}.
\end{equation}
The first summand is bounded by 2 as in the proof of Proposition \ref{kontr_spn}.
Since $w_{2,p}=w_{p,2}^{-1}$ we get by Lemma \ref{est_wpq} ii)
$$
|\tilde{\beta}(x)|=|w_{2,p}(x)|^{p-1}|w_{2,p}'(x)|=
\frac{|w_{2,p}(x)|^{p-1}}{w_{p,2}'(w_{2,p}(x))}
\leq \frac{8}{\sqrt{p}}|w_{2,p}(x)|^{p/2},
$$
hence
$$
\Big(\sum_{i=1}^n\tilde{\beta}^{2}(x_i)\Big)^{1/2}\leq \frac{8}{\sqrt{p}} 
\tilde{s}^{p/2}.
$$
Using (\ref{est_alpha}) and $\|W_{2,p}^n(x)\|_2\leq n^{1/2-1/p}\tilde{s}$ we bound
the second summand in (\ref{est_dtspn}) by
$$
\tilde{s}^{-p}\min\Big\{1,\frac{2p\tilde{s}^p}{n}\Big\}n^{1/2-1/p}\tilde{s}
\frac{8}{\sqrt{p}} \tilde{s}^{p/2}=
8p^{-1/p}\min\{u^{-1/2},2u^{1/2}\}u^{1/p}\leq 16,
$$
where $u:=p\tilde{s}^p/n$.
\end{proof}

\subsection{Applying $\nu_1$ results -- $p \leq 2$}

We start with the version of Theorem \ref{PushAndPop} for $\nu_p$.

\begin{lem} 
\label{LpPushAndPop} 
For any $A \in\BB(\R^n)$,
$p \in [1,2]$ and $t \geq 1$ at least one of the following holds:
\begin{itemize}
\item $\npn(A + 20 t^{1\slash p} B_p^n) \geq e^t \npn(A)$ or
\item $\npn\big((A + 20 t^{1\slash p} B_p^n) \cap 100 \sqrt{n} B_2^n \big)
\geq \frac{1}{2} \npn(A)$.
\end{itemize}
\end{lem}

\begin{proof} 
We will use the transport $W_{1,p}^n$ from $\nn$ to $\npn$. 
Proposition \ref{Wproperties} v) gives 
$\|W_{1,p}^n(x) - W_{1,p}^n(y)\|_p^p \leq 2^p \|x-y\|_1.$ 
By Remark \ref{normComparisons} this means that 
$A + 2(10t)^{1\slash p} B_p^n \supset W_{1,p}^n (W_{p,1}^n(A) + 10 t B_1^n)$. 
Let us fix $t\geq 1$ and apply Theorem \ref{PushAndPop} to $W_{p,1}^n(A)$ and 
$10 t$. If the second case occurs, we have
\begin{align*}
%%\npn\big(A + (20 \slash \cwlow) t^{1\slash p} B_p^n\big) & \geq 
\npn(A + 20 t^{1\slash p} B_p^n) &\geq 
\npn \big(W_{1,p}^n (W_{p,1}^n(A) + 10tB_1^n)\big)  
= \nn (W_{p,1}^n(A) + 10 t B_1^n) 
\\
&\geq e^{t} \nn(W_{p,1}^n(A)) = e^{t} \npn(A).
\end{align*}

If the first case of Theorem \ref{PushAndPop} occurs, then due to Proposition
\ref{Wproperties} iii) we have $\|W_{1,p}^n(x)\|_2 \leq 2\|x\|_2$, so 
$2\alpha B_2^n \supset W_{1,p}^n(\alpha B_2^n)$ for any $\alpha>0$. Thus
\begin{align*} 
%%\npn \big((A + (20 \slash \cwlow) t^{1\slash p} B_p^n) \cap 100\sqrt{n} B_2^n\big) 
%%& \geq
\npn \big((A + 20t^{1/p} B_p^n) \cap &100 \sqrt{n} B_2^n\big) 
 \geq
\npn\big(W_{1,p}^n(W_{p,1}^n(A) +  10tB_1^n) \cap 100 \sqrt{n} B_2^n\big) 
\\ 
& =
\npn\big(W_{1,p}^n\big((W_{p,1}^n(A) + 10tB_1^n)\cap W_{p,1}^n(100 \sqrt{n} B_2^n)
\big)\big) 
\\ 
& \geq
\npn\big(W_{1,p}^n\big((W_{p,1}^n(A)+10tB_1^n)\cap 50\sqrt{n}B_2^n\big)\big)
\\ 
& = 
\nn\big(\big(W_{p,1}^n(A) +  10tB_1^n\big) \cap 50\sqrt{n}B_2^n\big) 
\\ 
& \geq 
\frac{1}{2} \nn(W_{p,1}^n(A)) = \frac{1}{2} \npn(A).
\end{align*}
\end{proof}

\begin{lem} 
\label{Magia} 
There exists a constant $C$ such that for any $p \in [1,2]$, $t > 0$ and $n\in\N$ 
we have 
$$
\npn\big(A + C(t^{1\slash p} B_p^n + t^{1\slash 2} B_2^n)\big) 
\geq \min\Big\{\frac{1}{2},e^t \npn(A)\Big\}.
$$
\end{lem}

\begin{proof}
Corollary \ref{KsztaltBpNu} gives $B_s(\npn) \subset C (s^{1\slash p} 
B_p^n + s^{1\slash 2} B_2^n)$ for $s>0$. By Corollary \ref{SIClogprod}, 
$\npn$ satisfies $IC(48)$, which, due to Proposition \ref{ICconc}  
implies $\nu_p^n(A +48B_{2t}(\npn)) \geq \min\{1\slash 2, e^t\npn(A)\}$ 
for any Borel set $A$. Thus we have 
$$
\npn \big(A + 96 C (t^{1\slash p} B_p^n + t^{1\slash 2} B_2^n)\big) 
\geq \min\{1\slash 2,e^t \npn(A)\}.
$$
\end{proof}

\begin{prop} 
\label{malaNorma}
For any $\alpha > 1$ there exists a constant $c(\alpha)$ such that for any $n\in \N$ 
and $p\geq 1$ we have 
$$
\npn(\{x : \|x\|_p < c(\alpha) n^{1\slash p}\}) < \alpha^{-n}.
$$
\end{prop}

\begin{proof} 
We have 
\begin{align*} 
\npn(\{x : \|x\|_p < c(\alpha) n^{1\slash p}\}) = & 
\frac{n}{\Gamma(1 + \frac{n}{p})} \int_0^{c(\alpha) n^{1/p}} e^{-r^p} r^{n-1}dr  
\\ 
\leq &
\frac{n}{\Gamma(1+\frac{n}{p})} \int_0^{c(\alpha) n^{1\slash p}} r^{n-1} 
= \frac{c(\alpha)^n n^{n\slash p}}{\Gamma(1+\frac{n}{p})} 
\leq (C c(\alpha))^n,\end{align*}
where in the last step we use the Stirling approximation and $C$ as always 
denotes a universal constant. Thus it is enough to take 
$c(\alpha) < (C\alpha)^{-1}$.
\end{proof}

\begin{thm} 
\label{glownyWynikMaleP} 
There exists a universal constant $C$ such that $\mu_{p,n}$ satisfies $\CI(C)$
for any $p \in [1,2]$ and  $n\in\N$.
\end{thm}

\begin{proof}
By Propositions \ref{equiv}, \ref{small_sets}, \ref{B_sub_Z} 
and \ref{ksztaltLambdaStarMi} it is enough to show
\begin{equation}
\label{est_smallp}
\mpn \big(A + C \big(t^{1\slash p} B_p^n+ t^{1\slash 2} B_2^n\big)\big) 
\geq \min\{1\slash 2,e^t \mpn(A)\}.
\end{equation}
for  $1\leq t\leq n$ and $\mu_{p,n}(A)\geq e^{-n}$.

Recall that $T_{p,n}$ denotes the map transporting $\npn$ to $\mpn$. Apply 
Lemma \ref{LpPushAndPop} to $T_{p,n}^{-1}(A)$ and $t$. If the first case occurs, 
we have 
$$
\npn\big(T_{p,n}^{-1}(A)+20 t^{1\slash p}B_p^n\big) 
\geq e^t \npn\big(T_{p,n}^{-1}(A)\big) = e^t \mpn(A).
$$
Proposition \ref{kontrakcjaLp} gives $\|T_{p,n}x-T_{p,n}y\|_p\leq 2\|x - y\|_p$,
thus by Remark \ref{normComparisons}, 
\begin{align*}
\mpn\big(A + 40 t^{1\slash p} B_p^n\big)
&= \npn\big(T_{p,n}^{-1}\big(A + 40 t^{1\slash p} B_p^n\big)\big)
\\
&\geq \npn\big(T_{p,n}^{-1}(A) + 20 t^{1\slash p}B_p^n\big)
\geq e^t \mpn(A)
\end{align*}
and we obtain (\ref{est_smallp}) in this case.

Hence we may assume that the second case of Lemma \ref{LpPushAndPop} holds,
that is
$$
\npn(A') \geq \frac{1}{2}\npn(T_{p,n}^{-1}(A)) = \frac{1}{2}\mpn(A),
$$ 
where
$$
A':=\big(T_{p,n}^{-1}(A)+ 20 t^{1\slash p} B_p^n\big)\cap 100\sqrt{n}B_2^n.
$$
In particular $\npn(A')\geq e^{-n}/2$. Let
$$
A'':= A'\cap \{x : \|x\|_p \geq \tilde{c} n^{1\slash p}\},
$$
where $\tilde{c}=c(4e)$ is a constant given by 
Proposition \ref{malaNorma} for $\alpha=4e$. Then
$$
\npn(A'')\geq \npn(A')-(4e)^{-n}\geq \frac{1}{2}\npn(A')\geq \frac{1}{4}\mpn(A).
$$
We apply Lemma \ref{Magia} for $A''$ and $4t$ to get 
\begin{align*}
\mpn\big(T_{p,n}&\big(A''+4C\big(t^{1/p}B_p^n+t^{1/2} B_2^n\big)\big)\big)
%%=& \npn\big(A'' + 4C \big(t^{1\slash p} B_p^n + t^{1\slash 2} B_2^n\big)\Big) \\
\geq  \npn\big(A''+C\big((4t)^{1/p} B_p^n + (4t)^{1/2} B_2^n\big)\big) 
\\
&\geq  \min\Big\{\frac{1}{2},e^{4t} \npn(A'')\Big\}
\geq \min\Big\{\frac{1}{2},e^{4t} \frac{\mpn(A)}{4}\Big\} 
\geq \min\Big\{\frac{1}{2},e^t \mpn(A)\Big\}.
\end{align*}

Proposition \ref{kontrakcjaLp} and Remark \ref{normComparisons} imply
$$
T_{p,n}\big(A''+4C t^{1/2}B_2^n+4C t^{1/p}B_p^n\big) 
\subset T_{p,n}\big(A'' + 4C t^{1/2} B_2^n\big)
+8C t^{1/p} B_p^n.
$$
Moreover, for $x \in A''$ we have  $\|x\|_2 \leq 100 \sqrt{n}$ and 
$\|x\|_p \geq \tilde{c} n^{1/p}$. Thus $n^{-1/2}\|x\|_2\leq 
100\tilde{c}^{-1} n^{-1/p} \|x\|_p$, so we can use Proposition
\ref{kontrakcjaL2} along with Remark \ref{normComparisons} to get 
$$
T_{p,n}(A''+4C t^{1/2}B_2^n) \subset T_{p,n}(A'') +\tilde{C} t^{1/2} B_2^n.
$$

Proposition \ref{kontrakcjaLp}, Remark \ref{normComparisons} 
and the definitions of $A'$ and $A''$ yield
$$
T_{p,n}(A'') \subset T_{p,n}(A') \subset 
T_{p,n}\big(T_{p,n}^{-1}(A) + 20 t^{1/p} B_p^n\big) 
\subset A + 40 t^{1\slash p} B_p^n.
$$
Putting the four estimates together, we can write 
\begin{align*} 
\mpn \Big(A +&(40+8C)t^{1/p} B_p^n 
+\tilde{C} t^{1/2} B_2^n\Big)
\\
&\geq \mpn \Big(T_{p,n}(A'')+\tilde{C}t^{1/2}B_2^n + 8Ct^{1/p} B_p^n\Big) 
\\
&\geq \mpn \Big(T_{p,n}\big(A''+4C t^{1/2} B_2^n\big) + 8Ct^{1/p} B_p^n\Big) 
\\
&\geq \mpn \big(T_{p,n}\big(A'' + 4C \big(t^{1/p}B_p^n+t^{1/2} B_2^n\big)\big)\big)
\geq \min\Big\{\frac{1}{2},e^t \mpn(A)\Big\},
\end{align*} 
which gives (\ref{est_smallp}) in the second case and ends the proof.
\end{proof}

\begin{cor} 
\label{ICsmallP} 
There exists an absolute constant $C$ such that the measure $\mi_{p,n}$ 
satisfies $\IC(C)$ for any $p \in [1,2]$ and $n \in \N$. 
\end{cor}

\begin{proof} 
Let
$$
\tilde{\mu}_{p,n}(A):=\mu_{p,n}(\sigma_{p,n}A),
\mbox{ where }
\sigma_{p,n}^2:=\int x_1^2d\mu_{p,n},
$$
then $\tilde{\mu}_{p,n}$ is isotropic. Both properties
$\IC$ and $\CI$ are affine invariant, so by Theorem \ref{glownyWynikMaleP},
$\tilde{\mu}_{p,n}$ has property $\CI(C)$ and we are to show that it
satisfies $\IC(C)$. By Corollary \ref{IC_and_CI} we only
need to show Cheeger's inequality for $\tilde{\mu}_{p,n}$ with uniform
constant.

A recent result of S.\ Sodin (\cite[Theorem 1]{So}) 
states (after rescaling from $B_p^n$ to $\cpn B_p^n$) that 
\begin{equation} 
\label{sodinPoincare} 
\mu_{p,n}^+(A) \geq c \min\{\mpn(A),1-\mpn(A)\} \log^{1-1/p}
\frac{1}{\min\{\mpn(A),1-\mpn(A)\}}
%%\geq c \min\{\mpn(A),1-\mpn(A)\}
\end{equation} 
for some universal constant $c$. 
Thus $\tilde{\mu}_{p,n}$ satisfies Cheeger's inequality with constant
$c\sigma_{n,p}$ and it is enough to notice that by (\ref{comp_mom}),
$$
\sigma_{p,n}\sim \frac{\cpn}{n^{1/p}}\Big(\int x_1^2d\nu_p(x)\Big)^{1/2}\sim 1.
$$
\end{proof}

\subsection{The easy case -- $p \geq 2$}

This case will follow easily from the exponential case and the facts from 
subsection \ref{Transporty}. 

\begin{thm} 
\label{conc_largeP} 
There exists a universal constant $C$ such that for any $A \subset \R^n$, 
any $t \geq 1$, $n \geq 1$ and $p \geq 2$ we have 
$$
\mpn\Big(A + C \big(t^{1\slash p} B_p^n \cap t^{1\slash 2} B_2^n\big)\Big) 
\geq \min\{1\slash 2,e^t \mi_p(A)\}.$$
\end{thm}

\begin{proof}
In this case we will again use the transport $S_{p,n}$. Assume 
$A \subset \cpn B_p^n$, let $\tilde{A} := S_{p,n}^{-1}(A)$. By Talagrand's
inequality (\ref{conc_nu}) we have 
$\nu^n(\tilde{A} + Ct B_1^n + \sqrt{Ct} B_2^n) \geq 
\min\{e^t \nu^n(\tilde{A}), 1/2\}$. 
However by Corollary \ref{KontrakcjaDuzep} we have 
$$
S_{p,n}\Big(\tilde{A} + Ct B_1^n + \sqrt{Ct} B_2^n)\Big) 
\subset S_{p,n}(\tilde{A}) + 10C \big(\sqrt{t} B_2^n + t^{1\slash p} B_p^n\big).
$$ 
Thus, as $S_{p,n}(\tilde{A})=A$ and $S_{p,n}$ transports the measure $\nu^n$ 
to $\mpn$, we get the thesis.
\end{proof}

By Propositions \ref{equiv}, \ref{small_sets}, \ref{B_sub_Z} 
and \ref{ksztaltLambdaStarMi}, Theorem \ref{conc_largeP} yield the following.

\begin{cor}
\label{CI_largeP}
There exists an absolute constant $C$ such that $\mpn$ satisfies $\CI(C)$
for $p\geq 2$.
\end{cor}

\begin{thm} 
\label{Cheeger_largeP}
For any $p \geq 2$ and $n \geq 1$ the 
measure $\mpn$ satisfies Cheeger's inequality (\ref{cheeger}) with the constant 
$1\slash 20$. 
\end{thm}

\begin{proof} Again we shall transport this result from the exponential measure.
By \cite{BH1} Cheeger's inequality holds for $\nu^n$ with the constant 
$\kappa=1/(2\sqrt{6})$, thus by Proposition  \ref{kontr_spn} $\mpn$ satisfies 
(\ref{cheeger}) with the constant $\kappa/4\geq 1/20$.
\end{proof}

As in the proof of Corollary \ref{ICsmallP} we show that
Theorem \ref{Cheeger_largeP} and Corollary \ref{CI_largeP} imply
infimum convolution inequality for $\mpn$, $p\geq 2$. Adding the two results together
we get

\begin{thm}
\label{glownyWynik} 
There exists a universal constant $C$ such that for any $p \in [1,\infty]$ and any 
$n \in \N$ the measure $\mpn$ satisfies $\IC(C)$.
\end{thm} 

We conclude this section with the proof of logaritmic Sobolev--type inequality
for $\mpn$.

\begin{thm} \label{logSobolev}
Let $\Phi(x)=(2\pi)^{-1/2}\int_{x}^{\infty}\exp(-y^2/2)$ be a Gaussian distribution
function, $A\in \BB(\er^n)$  and $p\geq 2$, then
$$
\mpn(A)=\Phi(x)\ \Rightarrow\ \mpn(A+18\sqrt{2}tB_2^n)\geq \Phi(x+t)
\mbox{ for all } t>0.
$$
In particular there exists a universal constant $C$ such that 
$$
\mpn^{+}(A)\geq \frac{1}{C}\min\bigg\{\mpn(A)\sqrt{\ln\frac{1}{\mpn(A)}},
(1-\mpn(A))\sqrt{\ln\frac{1}{1-\mpn(A)}}\bigg\}.
$$
\end{thm}

\begin{proof}
By Proposition \ref{kontr_tildespn}, $\tilde{S}_{p,n}(\sqrt{2}\cdot)$
is $18\sqrt{2}$--Lipschitz and transports
the canonical Gaussian measure on $\er^n$ onto $\mpn$. Hence the first part of
theorem follows by the Gaussian isoperimetric inequality of Borell \cite{Bo2}
and Sudakov, Tsirel'son \cite{ST}. The last estimate immediately follows
by a standard estimate of the Gaussian isoperimetric function.
\end{proof}

\section{Concluding Remarks}

1. With the notion of the $\IC$ property one may associate $\IC$--domination 
of symmetric probability measures $\mu,\tilde{\mu}$ on $\er^n$: 
we say that  {\em $\mu$
is  $\IC$--dominated by $\tilde{\mu}$ with a constant $\beta$} if 
$(\mu,\Lambda_{\tilde{\mu}}^{*}(\frac{\cdot}{\beta}))$ has property $\tau$.
$\IC$--domination has the tensorization property: if $\mu_i$ are
$\IC(\beta)$--dominated by $\tilde{\mu}_i$, $1\leq i\leq n$,
then $\otimes \mu_i$ is $\IC(\beta)$--dominated by $\otimes \tilde{\mu}_i$. 
An easy modification of the proof of Corollary \ref{IC_and_CI} gives that
if $\mu$ is $\IC(\beta)$--dominated by an $\alpha$--regular measure
$\tilde{\mu}$, then
$$
\forall_{p\geq 2}\forall_{A\in {\cal B}(\er^n)}\
\mu(A)\geq \frac{1}{2}\ \Rightarrow\
1-\mu(A+c(\alpha)\beta{\cal Z}_{p}(\tilde{\mu}))\leq e^{-p}(1-\mu(A)).
$$
Following the proof of Proposition \ref{weakstr} we also get
for all $p\geq 2$,
$$
\Big(\int \big|\|x\| -\Med_{\mu}(\|x\|)\big|^pd\mu\Big)^{1/p}\leq
\tilde{c}(\alpha)\beta
\sup_{\|u\|_*\leq 1}\Big(\int |\is{u}{x}|^pd\tilde{\mu}\Big)^{1/p}.
$$

2. One may consider convex versions of properties $\CI$ and $\IC$.
We say that a symmetric probability measure $\mu$ satisfies the {\em convex
infimum convolution inequality with a constant $\beta$} if the pair
$(\mu,\Lambda_{\mu}^{*}(\frac{\cdot}{\beta}))$ has convex property $(\tau)$,
i.e. the inequality (\ref{in_infconv}) holds for all convex function
and $f$ with $\varphi(x)=\Lambda_{\mu}^{*}(x/\beta)$. Analogously
$\mu$ satisfies {\em convex concentration inequality with a constant $\beta$},
if (\ref{in_CI}) holds for all convex Borel sets $A$. We do not know if
convex $\IC$ implies convex $\CI$, however for $\alpha$--regular measures
it implies a weaker version of convex $\CI$, namely
$$
\mu(A)\geq 1/2 \ \Rightarrow\  \mu(A+c_1(\alpha)\beta {\cal Z}_{p}(\tilde{\mu}))
\geq 1-2e^{-p}
$$
and this property yields $\CWSM(c_2(\alpha)\beta)$.

From the results of \cite{Ma} one may easily deduce that the uniform 
distribution on $\{-1,1\}^n$ satisfies convex $\IC(C)$ with a universal
constant $C$.

\medskip

3. Property $\IC$ may be also investigated for nonsymmetric measures.
However in this case the natural choice of the cost function is
$\Lambda_{\tilde{\mu}}^{*}(x/\beta)$,
%%\Lambda_{\mu}^{*}(x/\beta)\Lambda_{\mu}^{*}(-x/\beta)$, 
where $\tilde{\mu}$ is the convolution of $\mu$ and the symmetric image of $\mu$.

\medskip

4. We do not know if the infimum convolution property (at least for regular measures)
implies Cheeger's inequality. If so, we would have equivalence of $\IC$ and $\CI$ + Cheeger.

\begin{tabbing}
11111111111111111111111111111111111111\=1111111111111111111111\=\kill
Rafa{\l} Lata{\l}a \>  Jakub Onufry Wojtaszczyk \\
Institute of Mathematics \> Institute of Mathematics\\
Warsaw University \> Warsaw University\\
Banacha 2, 02-097 Warszawa \> Banacha 2, 02-097 Warszawa\\
and \> Poland \\
Institute of Mathematics\> {\tt onufry@mimuw.edu.pl}\\
Polish Academy of Sciences\\
\'Sniadeckich 8\\
P.O.Box 21, 00-956 Warszawa 10\\
Poland\\
{\tt rlatala@mimuw.edu.pl}
\end{tabbing}

\end{document}